\documentclass[a4paper,reqno]{amsart}
\usepackage{amsmath}
\usepackage{amssymb}
\usepackage{amsthm}
\usepackage{latexsym}
\usepackage{amscd}
\usepackage{xypic}
\xyoption{curve}
\usepackage{ifthen} 
\usepackage{hyperref} 
\usepackage{graphicx}
\usepackage{enumerate}
\usepackage{enumitem}
\usepackage{rotating}
\usepackage{xcolor}
\usepackage{mathtools}
\usepackage{todonotes}

\usepackage{float}
\restylefloat{table}

\numberwithin{equation}{section}

\def\cocoa{{\hbox{\rm C\kern-.13em o\kern-.07em C\kern-.13em o\kern-.15em A}}}

\newtheorem{theorem}{Theorem}[section]
\newtheorem{question}[theorem]{Question}
\newtheorem{lemma}[theorem]{Lemma}
\newtheorem{proposition}[theorem]{Proposition}
\newtheorem{corollary}[theorem]{Corollary}

\theoremstyle{definition}
\newtheorem{remark}[theorem]{Remark}
\newtheorem{definition}[theorem]{Definition}
\newtheorem{example}[theorem]{Example}

\newtheorem{construction}[theorem]{Construction}

\newcommand {\sHom}{\mathcal{H}\kern -0.25ex{\mathit om}}
\newcommand {\sExt}{\mathcal{E}\kern -0.25ex{\mathit xt}}
\newcommand {\sTor}{\mathcal{T}\kern -0.25ex{\mathit or}}
\newcommand {\im}{\mathrm{im}}
\newcommand {\rk}{\mathrm{rk}}

\newcommand {\Ext}{\mathrm{Ext}}
\newcommand {\Hom}{\mathrm{Hom}}
\newcommand {\Hilb}{\mathcal{H}\kern -0.25ex{\mathit ilb\/}}
\newcommand {\Mov}{\overline{\mathrm{Mov}}}

\newcommand {\cK}{\mathcal{K}}
\newcommand {\cA}{\mathcal{A}}

\newcommand {\cU}{\mathcal{U}}
\newcommand {\cV}{\mathcal{V}}

\newcommand {\cP}{\mathcal{P}}
\newcommand {\cQ}{\mathcal{Q}}
\newcommand{\cC}{{\mathcal C}}
\newcommand{\cS}{{\mathcal S}}
\newcommand{\cE}{{\mathcal E}}
\newcommand{\cF}{{\mathcal F}}
\newcommand{\cM}{{\mathcal M}}
\newcommand{\cN}{{\mathcal N}}
\newcommand{\cO}{{\mathcal O}}
\newcommand{\cG}{{\mathcal G}}

\newcommand{\cI}{{\mathcal I}}

\newcommand {\bZ}{\mathbb{Z}}

\newcommand {\bC}{\mathbb{C}}
\newcommand {\bP}{\mathbb{P}}

\newcommand{\Bl}{\operatorname{Bl}}
\newcommand{\Pic}{\operatorname{Pic}}

\def\p#1{{\bP^{#1}}}

\def\mapright#1{\mathbin{\smash{\mathop{\longrightarrow}
\limits^{#1}}}}

\title[Instanton bundles]{Instanton bundles on the blow up of\\ the projective $3$--space at a point}

\thanks{The first and fourth authors are members of GNSAGA group of INdAM, are supported by the framework of PRIN 2015 \lq Geometry of Algebraic Varieties\rq, cofinanced by MIUR and  by MIUR grant Dipartimenti di Eccellenza 2018-2022 (E11G18000350001). The second author was supported by TUBITAK project 114F116. }

\subjclass[2010]{Primary: 14J60. Secondary: 14J45, 14D21, 14F05.}

\keywords{Fano threefold, del Pezzo threefold, vector bundle, $\mu$--(semi)stable bundle, instanton bundle}

\author[G. Casnati, E. Coskun, O. Genc, F. Malaspina]{G. Casnati, E. Coskun, O. Genc, F. Malaspina }

\begin{document}

\begin{abstract}
We propose a general definition of mathematical instanton bundle with given charge on any Fano threefold extending the classical definitions on $\p3$ and on Fano threefold with cyclic Picard group. Then we deal with the case of the blow up of $\p3$ at a point, giving an explicit construction of instanton bundles satisfying some important extra properties: moreover, we also show that they correspond to smooth points of a component of the moduli space.
\end{abstract}

\maketitle

\section{Introduction}
Let $\p N$ be the projective space of dimension $N$ over the complex field $\bC$. We recall that a smooth irreducible closed subscheme $X\subseteq \p N$ of dimension $3$ is called a {\sl Fano threefold} if the dual of its canonical bundle $\omega_X$ is ample: we refer the interested reader to \cite{I--P} for the results about Fano threefold mentioned in what follows. 

The greatest positive integer $i_X$ such that $\omega_X\cong\cO_X(-i_Xh)$ for some ample line bundle $\cO_X(h)\in\Pic(X)$ is called the {\sl index} of $X$ and $\cO_X(h)$ the {\sl fundamental line bundle of $X$}. One has  $1\le i_X\le 4$ and $i_X=4,3$ if and only if $X$ is isomorphic to either $\p3$, or a smooth quadric in $\p4$, respectively. Up to isomorphism there exist exactly $8$ deformation families of Fano threefolds with $i_X=2$ and $95$ deformation families of Fano threefolds with $i_X=1$ (see \cite[Chapter 20]{I--P}).

The study of vector bundles on Fano threefolds is an important research field, even in the very first cases of the projective space and of the smooth quadric. A particularly important family of rank $2$ vector bundles on Fano threefolds is related to physics. In order to introduce such a family we fix some notation and recall some facts below. 

Let $X$ be any threefold endowed with an ample line bundle $\cO_X(H)$. Recall that for each sheaf $\cF$ on $X$ its {\sl slope} and {\sl reduced Hilbert polynomial} with respect to $\cO_X(H)$ are
$$
\mu(\cF)= c_1(\cF)H^2/\rk(\cF), \qquad p_{\cF}(t)=\chi(\cF(tH))/\rk(\cF),
$$
respectively.

We say that the bundle $\cF$ is {\sl $\mu$--stable} (resp. {\sl $\mu$--semistable}) with respect to $\cO_X(H)$ if  $\mu(\mathcal G) < \mu(\cF)$ (resp. $\mu(\mathcal G) \le \mu(\cF)$) for each subsheaf $\mathcal G$ with $0<\rk(\mathcal G)<\rk(\cF)$. The bundle $\cF$ is called {\sl stable} (resp. {\sl semistable})  with respect to $\cO_X(H)$ if $p_{\mathcal G}(t)<  p_{\cF}(t)$ (resp. $p_{\mathcal G}(t)\le  p_{\cF}(t)$)  for each proper subsheaf $\mathcal G$. We have the following chain of implications for $\cF$:
$$
\text{$\cF$ is $\mu$--stable}\ \Rightarrow\ \text{$\cF$ is stable}\ \Rightarrow\ \text{$\cF$ is semistable}\ \Rightarrow\ \text{$\cF$ is $\mu$--semistable.}
$$

It is easy to see that a vector bundle  $\cF$ is ($\mu$--semi)stable with respect to $\cO_X(H)$ if and only if it is ($\mu$--semi)stable with respect to $\cO_X(\lambda H)$ for each positive integer $\lambda$.
There exists a quasi--projective scheme $\cM_X(r;c_1,c_2)$ parameterizing semistable vector bundles on $X$ with respect to $\cO_X(H)$ with fixed rank $r$ and Chern classes $c_1,c_2$. 

Now let $X$ be a Fano threefold of index $i_X$, define
$$
q_X:=\left[\frac {i_X}2\right]
$$
and assume that the Picard number $\varrho_X$ of $X$ is $1$. In \cite{Fa} (see also \cite{Kuz} in the case $i_X=2$: see also \cite{Sa} for a more detailed analysis when $\deg(X)=5$), the author calls instanton bundle on such a Fano threefold $X$ (sometimes also called mathematical instanton bundle) every stable rank $2$ vector bundle $\cE$ such that $c_1(\cE)=(2q_X-i_X)h$ and $h^1\big(X,\cE(-q_Xh)\big)=0$. Such a definition coincides with the classical one when $X=\p3$: see \cite{O--S--S} and the references therein. In \cite{Fa, Kuz}, several new examples are described, showing the existence of instanton bundles on some Fano threefolds $X$ with $\varrho_X=1$. 

In \cite{M--M--PL} the authors slightly weakened the above definition following the classical papers \cite{Don, Buch}, in order to deal with bundles on a Fano variety with Picard number $\varrho_X\ge2$. More precisely, let $X$ be the {\sl flag threefold}, i.e. the general hyperplane section $X$ of the Segre embedding $\p2\times\p2\subseteq\p8$: notice that $i_X=2$ in this case. 

The authors call instanton bundle  every $\mu$--semistable rank $2$ bundle $\cE$ on $X$ such that $c_1(\cE)=0$ and $h^0\big(X,\cE\big)=h^1\big(X,\cE(-h)\big)=0$, proving their existence and classifying the strictly $\mu$--semistable ones for each admissible value of $c_2(\cE)$. Such a classification allows them to prove that the notion of $\mu$--semistability is actually equivalent to the semistability one for instanton bundles on the flag threefold $X$.

A similar definition is also adopted in \cite{An}, where instanton bundles on the Segre embedding $\p1\times\p1\times\p1\subseteq\p7$ are described.

In the present paper, partially motivated by the above discussion, we adopt the latter definition of instanton bundle, extending it to any Fano threefold regardless of its Picard number. 

\begin{definition}
\label{dInstanton}
Let $X$ be a Fano threefold.

A vector bundle $\cE$ of rank $2$ on $X$ is called an instanton bundle if the following properties hold:
\begin{itemize}
\item $c_1(\cE)=(2q_X-i_X)h$;
\item $\cE$ is $\mu$--semistable and $h^0\big(X,\cE\big)=0$;
\item $h^1\big(X,\cE(-q_Xh)\big)=0$;
\end{itemize}
\end{definition}

Following the usual terminology, if $\cE$ is an instanton bundle on a Fano threefold $X$, then the class $c_2(\cE)$ will be called the {\sl charge} of $\cE$.
Moreover, if $\cE$ is an instanton bundle we will refer to the vanishing of $h^1\big(X,\cE(-q_Xh)\big)$ as {\sl the instantonic condition}.

We now turn our attention to two other interesting properties of instanton bundles related to their semistability. Notice that $\mu$--stability implies both $\mu$--semistability and $h^0\big(X,\cE\big)=0$, and it is actually equivalent to the latter vanishing if $i_X$ is even and $\varrho_X=1$.

We first notice that if $\varrho_X=1$, then for each $d\ge1$, the line bundle $\cO_{X}(dh)$ is very ample, hence $\cE\otimes\cO_D$ is $\mu$--semistable for each general  $D\subseteq X$ with $D\in\vert dh\vert$, thanks to \cite[Theorem 3.1]{Ma}. In particular we certainly have $h^0\big(D,\cE(-q_Xh)\otimes\cO_D\big)=0$. The cohomology of the exact sequence
\begin{equation}
\label{seqRestriction}
0\longrightarrow \cO_{X}(-D)\longrightarrow \cO_{X}\longrightarrow \cO_D\longrightarrow 0
\end{equation}
tensored by $\cE(-q_Xh)$ gives $h^1\big({X},\cE(-q_Xh-D)\big)=h^0\big(D,\cE(-q_Xh)\otimes\cO_D\big)=0$. Hence, a $\mu$--stable rank $2$ bundle $\cE$ with $c_1(\cE)=(2q_X-i_X)h$ on $X$ is an instanton bundle  if and only if  $h^1\big(X,\cE(-ih)\big)=0$ for each $i\ge q_X$. In Section \ref{sInstanton} we will show that one can still prove similarly that $h^1\big(X,\cE(-h-D)\big)=0$ for each effective smooth divisor $D$ on the flag threefold $X$.

As a second interesting property we recall that a subscheme $L\subseteq X$ is called a line if $L\cong\p1$ and $Lh=1$: we denote by $\Lambda$ the Hilbert scheme of lines on $X$. The restriction of an instanton bundle on $\p3$ to each general line $L$ is the trivial bundle $\cO_{\p1}^{\oplus2}$, thanks to the Grauert--M\"ulich theorem (see \cite[Corollary 2 of Theorem II.2.1.4]{O--S--S}). Similarly, for $i_X=3$ the restriction is $\cO_{\p1}{\oplus}\cO_{\p1}(-1)$: see \cite[Proposition 5.3]{C--F}.  

The problem of understanding how the restriction of an instanton bundle to general hyperplanes and lines behaves seems to be of particular interest. 
For this reason we give the following further definitions. 

\begin{definition}
\label{dEarnest}
Let $\cE$ be an instanton bundle on a Fano threefold $X$.
\begin{itemize}
\item We say that $\cE$ is generically trivial (resp. on the component $\Lambda_0\subseteq\Lambda$) if $h^1\big(L,\cE((i_X-2q_X-1)h)\otimes\cO_L\big)=0$ when $L\in\Lambda$ (resp. on $\Lambda_0$) is general.
\item We say that $\cE$ is  earnest if $h^1\big(X,\cE(-q_Xh-D)\big)=0$ when $\vert D\vert\ne\emptyset$ contains smooth integral elements.
\end{itemize}
\end{definition}

Notice that $h^1\big(L,\cE((i_X-2q_X-1)h)\otimes\cO_L\big)=0$ if and only if either $\cE\otimes\cO_L\cong\cO_{\p1}^{\oplus2}$ (when $i_X$ is even), or $\cE\otimes\cO_L\cong\cO_{\p1}(-1)\oplus\cO_{\p1}$ (when $i_X$ is odd). The locus of lines $L\in\Lambda$ with such a splitting property is open by semicontinuity. Thus $\cE$ is generically trivial on a component $\Lambda_0$ if and only if there is at least one $L\in\Lambda_0$ with such a splitting property.

As we pointed out above each instanton bundle is generically trivial if $i_X\ge3$. When $i_X\le2$ the problem of the generic triviality of each instanton bundle is open. Indeed, in this case, it seems difficult to extend the proof of the aforementioned Grauert--M\"ulich theorem, because the natural map $\mathcal L\to X$ from the universal line $\mathcal L\to\Lambda$ has no more connected fibres (see \cite[Section 3.7 and Conjecture 3.16]{Kuz}). 

Let $E$ be a  {\sl plane}, i.e. a divisor $E\subseteq X$ such that $E\cong\p2$ and $Eh^2=1$. The universal property of the Hilbert scheme yields the existence of a morphism $\epsilon\colon E^\vee\cong\p2\to\Lambda$ whose image contains all the points corresponding to lines in $E$. Thus there exists a component $\Lambda_E\subseteq\Lambda$ containing $\im(\epsilon)$. In Section \ref{sGeneral} we will show that the vanishing $h^1\big(X,\cE(-q_Xh-E)\big)=0$  is related to the generic triviality on the component $\Lambda_E$. 

In Section \ref{sInstanton} we also show that the above definitions lead to the well known classical examples. 
Notice that in these classical examples, essentially thanks to the $\mu$--semistability hypothesis, one obtains that $c_2(\cE)\in\Mov(X)$, i.e. it satisfies $c_2(\cE)D\ge0$ for each element $D$ of the pseudo--effective cone of $X$  (for further details on $\Mov(X)$ see \cite[Section 11,4.C]{Laz2}): in particular such an inequality holds for every effective divisors on $X$.

In Section \ref{sGeneral} we show that an instanton bundle in the sense of the above definition has some interesting properties. E.g.
\begin{equation*}
c_2(\cE) h\ge\left\lbrace\begin{array}{ll} 
1\quad&\text{if $i_X=4$,}\\
2\quad&\text{if $i_X=2,3$,}\\
\frac{\deg(X)}4\quad&\text{if $i_X=1$.}
\end{array}\right.
\end{equation*}
Moreover, if $2\le i_X\le 3$ and the fundamental line bundle gives an embedding $X\subseteq\p N$, then instanton bundles $\cE$ of minimal charge are {\sl Ulrich bundles} up to twisting them by $\cO_X(h)$, i.e. 
\begin{gather*}
h^i\big(X,\cE((1-i)h)\big)=0,\qquad i>0,\\
h^j\big(X,\cE(-jh)\big)=0,\qquad j<3.
\end{gather*}

In Section \ref{sBlow} we begin to deal with the case of the Fano threefold $F\subseteq \p8$ of degree $7$, i.e. the blow up $F:=\Bl_P\p3$ of $\p3$ at a point $P$. We have a natural isomorphism $F\cong\bP(\cP)\mapright\pi\p2$, where $\cP:=\cO_{\p2}\oplus\cO_{\p2}(1)$. We denote by  $\xi$ and $f$ the classes in $A(F)$ of the line bundles $\cO_{\bP(\cP)}(1)$ and $\pi^*\cO_{\p2}(1)$ respectively. Thus we have an isomorphism
$$
A(F)\cong\bZ[\xi,f]/(f^3,\xi^2-\xi f),
$$
so that $\xi^3=\xi^2 f=\xi f^2$ are the classes of a point. The embedding $F\subseteq\p8$ is induced by the linear system $\cO_F(h)=\cO_F(\xi+f)$ and $\omega_F\cong\cO_F(-2h)$: in particular $i_F=2$ and its fundamental line bundle corresponds to the quadrics through the blown up point $P$. In what follows we will denote the exceptional divisor of the blow up by $E$.

If $\cE$ is an instanton bundle on $F$ with $c_2(\cE)=\alpha\xi^2+\beta f^2$, in the same section  we also construct a monad using particular full exceptional collections. We also use such a monad for proving some cohomological properties of instanton bundles on $F$. More precisely, we define the sheaves
\begin{gather*}
\cC^{-1}:=\cO_F(-f)^{\oplus\beta+\gamma}\oplus\pi^*\Omega_{F}^1(f-\xi)^{\oplus\alpha},\\
\cC^0:=\cO_F(-f)^{\oplus\gamma}\oplus\pi^*\Omega_{\p2}^1(f)^{\oplus\alpha+\beta}\oplus\cO_F(-\xi)^{\oplus2\alpha},\\
\cC^1:=\cO_F^{\oplus2\alpha+\beta-2},
\end{gather*}
Then the following theorem holds.

\begin{theorem}
\label{tSimplify}
Let $\cE$ be an instanton bundle with $c_2(\cE)=\alpha\xi^2+\beta f^2$ on $F$.

Then $\cE$ is the cohomology of a monad $\cC^\bullet$ of the form
\begin{equation*}
\label{Monad}
0\longrightarrow \cC^{-1}\longrightarrow \cC^0\longrightarrow\cC^1\longrightarrow0
\end{equation*}
where $\gamma:=h^1\big(F,\cE(-2\xi)\big)\le\alpha$.

Conversely, if the cohomology $\cE$ of the monad $\cC^\bullet$ is a $\mu$--semistable bundle for some integers $\alpha,\beta,\gamma$, then $\cE$ is an instanton bundle with charge $\alpha\xi^2+\beta f^2$ on $F$ such that
\begin{enumerate}
\item $h^1\big(F,\cE(-h-E)\big)\le\gamma$;
\item $h^1\big(F,\cE(-h-D)\big)=0$ for each smooth divisor $D\ne E$.
\end{enumerate}
\end{theorem}

Notice that
$$
h^1\big(F,\cE(-h-E)\big)=h^1\big(F,\cE(-2\xi)\big).
$$
Thus Theorem \ref{tSimplify} implies that if $\gamma=0$,
then $\beta\ge0$ thanks to  the definition of $\cC^{-1}$. We will see that this inequality is equivalent to $c_2(\cE)\in{\Mov}(F)$ (see Corollary \ref{cPositive}).

The following corollary is an immediate consequence of Theorem \ref{tSimplify} and of the short discussion above.

\begin{corollary}
\label{cSimplify}
Let $\cE$ be an instanton on $F$. Then $\cE$ is earnest if and only if
\begin{equation}
\label{Exotic}
h^1\big(F,\cE(-2\xi)\big)=0.
\end{equation}
Moreover, if this is true, then $c_2(\cE)\in{\Mov}(F)$,
\end{corollary}

Thus earnest instanton bundles on $F$ are actually characterised by a single condition. As we pointed out above, in the classical cases of Fano threefolds $X$ with $\varrho_X=1$ and of the flag threefold this is trivially true: it is then natural to ask the following question.

\begin{question}
Are earnest instanton bundles on a Fano threefold always characterised by a finite set of vanishings?
\end{question}

We conclude the section by characterising minimal  instanton bundles and giving some partial results on the $\mu$--(semi)stability of instanton bundles. In particular, we prove that if $\cE$ is an instanton bundle with charge $\alpha\xi^2+\beta f^2$ such that $\alpha\le14$, then $\cE$ is automatically $\mu$--stable. The following question is then natural.

\begin{question}
Is every instanton bundle on $F$ $\mu$--stable?
\end{question}

The problem of the existence of instanton bundles is studied in Section \ref{sProof}. More precisely, we describe a construction (see Construction \ref{conInstanton}) which leads to $\mu$--stable, earnest, generically trivial instanton bundles $\cE$ on $F$ with charge in $\Mov(F)$ and such that $\Ext^2_{F}\big(\cE,\cE\big)=0$, proving the following result.

\begin{theorem}
\label{tInstanton}
For each non--negative $\alpha,\beta\in\bZ$, with $2\alpha+\beta\ge2$, there is a generically trivial  earnest $\mu$--stable instanton bundle $\cE$ with charge $\alpha\xi^2+\beta f^2$ on $F$ such that 
$$
\dim\Ext^1_{F}\big(\cE,\cE\big)=8\alpha+4\beta-3,\qquad \Ext^2_{F}\big(\cE,\cE\big)=\Ext^3_{F}\big(\cE,\cE\big)=0.
$$
\end{theorem}

Let $\cI_F(\alpha\xi^2+\beta f^2) $ be the locus of points representing instanton bundles in the moduli space $\cM_F(2;0,\alpha\xi^2+\beta f^2)$.

Construction \ref{conInstanton} and the previous theorem have the following consequence whose proof is given in Section \ref{sComponent}.

\begin{theorem}
\label{tComponent}
For each non--negative $\alpha,\beta\in\bZ$, with $2\alpha+\beta\ge2$ there is an irreducible component 
$$
\cI_F^{0}(\alpha\xi^2+\beta  f^2)\subseteq\cI_F(\alpha\xi^2+\beta  f^2)
$$
 which is  generically smooth of dimension $8\alpha+4\beta-3$ and containing all the points corresponding to the bundles obtained via Construction \ref{conInstanton}.
 \end{theorem}
 
In the same section we also characterise instanton bundles with charge $\beta f^2$, showing that $\cI_F(\beta  f^2)$ is irreducible. Similarly, we also show therein that the locus of earnest instanton bundles  inside $\cI_F(\alpha\xi^2)$ coincides with $\cI_F^{0}(\alpha\xi^2)$.
 
\begin{question}
Is $\cI_F^{0}(\alpha\xi^2+\beta f^2)$  the locus of earnest instanton bundles  inside $\cI_F(\alpha\xi^2+\beta f^2)$?
\end{question}

 We are not able to prove or disprove the existence of instanton bundles on $F$ with charge in the complement of $\Mov(F)$. Thus the following related questions rise naturally.

\begin{question}
Is the charge of every instanton bundle on $F$ in $\Mov(F)$? Are there non--earnest instanton bundles?
\end{question}

\subsection{Acknowledgements}
The authors are particularly indebted with the referee for her/his criticisms, questions, remarks and suggestions which have considerably improved the whole exposition.

\section{General facts}
We list below some general facts which will be used in the paper.  Let $X$ be any smooth projective variety with canonical line bundle $\omega_X$.

If $\cG$ and $\mathcal H$ are coherent sheaves on $X$, then  the Serre duality holds
\begin{equation}
\label{Serre}
\Ext_X^i\big(\mathcal H,\cG\otimes\omega_X\big)\cong \Ext_X^{\dim(X)-i}\big(\cG,\mathcal H\big)^\vee
\end{equation}
(see \cite[Proposition 7.4]{Ha3}). 


In what follows we will handle simple sheaves. The following lemma will be helpful.

\begin{lemma}
\label{lExt3}
If $\cG$ is a simple coherent torsion--free sheaf with $c_1(\cG)=0$ on a smooth variety $X$ with $h^0\big(X,\omega_X^{-1}\big)>0$ and $h^0\big(X,\omega_X^{\rk(\cG)}\big)=0$, then
$$
\Ext^{\dim(X)}_{X}\big(\cG,\cG\big)=0.
$$
\end{lemma}
\begin{proof}
By definition $\dim\Hom_X\big(\cG,\cG\big)=1$. Thus every non--zero $\varphi\colon\cG\to\cG$ is the multiplication by a constant, hence it is injective because $\cG$ is torsion--free. Moreover, 
$$
\Ext^{\dim(X)}_{X}\big(\cG,\cG\big)\cong\Hom_X\big(\cG,\cG\otimes\omega_X\big)^\vee,
$$
by Equality \eqref{Serre}.

Every, non--zero $\psi\colon\cG\to\cG\otimes\omega_X$ induces a non--zero element $\varphi\in\Hom_X\big(\cG,\cG\big)$ by composing with the natural inclusion map $\cG\otimes\omega_X\subseteq\cG$. It follows that $\psi$ must be injective, hence $\det(\psi)\colon\det(\cG)\to\det(\cG)\otimes\omega_X^{\rk(\cG)}$ should be injective as well. Since $h^0\big(X,\omega_X^{\rk(\cG)}\big)=0$, it follows a contradiction. 
\end{proof}

Let $\cF$ be a vector bundle of rank $2$ on $X$ and let $s\in H^0\big(X,\cF\big)$. In general its zero--locus
$(s)_0\subseteq X$ is either empty or its codimension is at most
$2$. We can always write $(s)_0=S\cup Z$
where $Z$ has codimension $2$ (or it is empty) and $S$ has pure codimension
$1$ (or it is empty). In particular $\cF(-S)$ has a section vanishing
on $Z$, thus we can consider its Koszul complex 
\begin{equation}
  \label{seqSerre}
  0\longrightarrow \cO_X(S)\longrightarrow \cF\longrightarrow \cI_{Z\vert X}(-S)\otimes\det(\cF)\longrightarrow 0.
\end{equation}
Sequence \ref{seqSerre} tensored by $\cO_Z$ yields $\cI_{Z\vert X}/\cI^2_{Z\vert X}\cong\cF^\vee(S)\otimes\cO_Z$, whence the normal bundle of $Z$ inside $X$ satisfies
\begin{equation}
\label{Normal}
\cN_{Z\vert X}\cong\cF(-S)\otimes\cO_Z.
\end{equation}
If $S=0$, then $Z$ is locally complete intersection inside $X$, because $\rk(\cF)=2$. In particular, it has no embedded components.

The Serre correspondence allows us to revert the above construction as follows.

\begin{theorem}
  \label{tSerre}
  Let $Z\subseteq X$ be a local complete intersection subscheme of codimension $2$.
  
  If $\det(\cN_{Z\vert X})\cong\cO_Z\otimes\mathcal L$ for some $\mathcal L\in\Pic(X)$ such that $h^2\big(X,\mathcal L^\vee\big)=0$, then there exists a vector bundle $\cF$ of rank $2$ on $X$ such that:
  \begin{enumerate}
  \item $\det(\cF)\cong\mathcal L$;
  \item $\cF$ has a section $s$ such that $Z$ coincides with the zero locus $(s)_0$ of $s$.
  \end{enumerate}
  Moreover, if $H^1\big(X,{\mathcal L}^\vee\big)= 0$, the above two conditions  determine $\cF$ up to isomorphism.
\end{theorem}
\begin{proof}
See \cite{Ar}.
\end{proof}

The Riemann--Roch formula for a vector bundle $\cF$ on a threefold $X$ is
\begin{equation}
  \label{RRgeneral}
  \begin{aligned}
    \chi(\cF)&=\rk(\cF)\chi(\cO_F)+{\frac16}(c_1(\cF)^3-3c_1(\cF)c_2(\cF)+3c_3(\cF))-\\
    &-{\frac14}(\omega_Xc_1(\cF)^2-2\omega_Xc_2(\cF))+{\frac1{12}}(\omega_X^2c_1(\cF)+c_2(\Omega_{X}) c_1(\cF))
  \end{aligned}
\end{equation}
where $\omega_X:=c_1(\Omega_{X}))$ (see \cite[Theorem A.4.1]{Ha2}). Recall that  $i_Xhc_2(\Omega_{X})={24}$ and $\chi(\cO_F)=1$ when $X$ is a Fano threefold (see \cite[Exercise A.6.7]{Ha2}).

\section{Instanton bundles on Fano threefolds}
\label{sInstanton}
In this section we confront the classical definition of instanton bundle in \cite{Fa, Kuz, M--M--PL} with Definition \ref{dInstanton}.

Before dealing with the examples, we prove  the following lemma.

\begin{lemma}
\label{lExt12}
If $\cE$ is a simple instanton bundle on a Fano threefold $X$, then 
$$
\dim\Ext^1_{X}\big(\cE,\cE\big)-\dim\Ext^2_{X}\big(\cE,\cE\big)=2i_Xc_2(\cE) h-i_X\frac{c_1(\cE)^2h}2-3.
$$
\end{lemma}
\begin{proof}
The statement follows from Lemma \ref{lExt3} applied to $\cE$ and Equality \eqref{RRgeneral} for the bundle $\cE\otimes\cE^\vee$, taking into account that $\cE$ is simple. 
\end{proof}

Now we give some examples, showing that our definition either matches, or generalizes the different classical definitions of instanton bundle (e.g. see \cite{O--S--S,O--S,Fa,M--M--PL})

Recall that on a Fano threefold with $\Pic(X)$ generated by $\cO_X(h)$, a rank $2$ vector bundle $\cE$ on $X$ is   $\mu$--stable if and only if it is simple, if and only if $c_1(\cE)=0$ and $h^0\big(X,\cE\big)=0$, thanks to the Hoppe criterion (e.g. see \cite{Hop, J--M--P--S}).

\begin{example}
\label{eP3}
If $X\cong\p3$, then $i_X=4$ and $q_X=2$. In this case, our definition of instanton bundles coincides with the classical one. Moreover, they are all generically trivial and earnest, thanks to \cite[Corollary 2 of Theorem II.2.1.4]{O--S--S} and a theorem of Maruyama (see \cite[Theorem 3.1]{Ma}). 

It is well--known that instanton bundles $\cE$ on $\p3$ correspond to the point of an open subset $\cI(\alpha)\subseteq\cM_{\p3}(2;0,\alpha)$ where $\alpha\ge1$.

In order to see that $\cI(\alpha)$ is non--empty for each positive $\alpha\in\bZ$, let $A\subseteq\p3$ be the union of $\alpha+1$ pairwise skew lines: $A$ is non--degenerate in $\p3$ by construction, hence $h^0\big(\p3,\cI_{A\vert\p3}(1)\big)=0$. Moreover, $\cN_{A\vert \p3}\cong\cO_A\otimes\cO_{\p3}(1)^{\oplus2}$ because the lines in $A$ are pairwise disjoint. Thus Theorem \ref{tSerre} with $N=3$ implies the existence of a vector bundle $\cA$ with $c_1(\cA)=\cO_{\p3}(2)$ and $t\in H^0\big(\p3,\cA\big)$ such that $(t)_0=A$. The corresponding Koszul complex is
\begin{equation}
\label{seqSerreA}
0\longrightarrow\cO_{\p3}\longrightarrow\cA\longrightarrow\cI_{A\vert\p3}(2)\longrightarrow0.
\end{equation}

The cohomology of Sequence \eqref{seqSerreA} returns $h^0\big(\p3,\cA(-1)\big)\le h^0\big(\p3,\cI_{A\vert\p3}(1)\big)=0$ and $h^1\big(\p3,\cA(-3)\big)\le h^1\big(\p3,\cI_{A\vert\p3}(-1)\big)$. Trivially $h^0\big(A,\cO_A\otimes\cO_{\p3}(-1)\big)=0$, hence the cohomology of 
\begin{equation}
\label{seqStandard}
0\longrightarrow\cI_{A\vert\p3}\longrightarrow\cO_{\p3}\longrightarrow\cO_A\longrightarrow0,
\end{equation}
implies $h^1\big(\p3,\cI_{A\vert\p3}(-1)\big)=0$.
It follows that $\cU:=\cA(-1)$ is an instanton bundle on $\p3$ with $c_2(\cU)=c_2(\cA(-1))=\alpha$.  

In particular $\cU$ is $\mu$--stable. Moreover, the cohomology of Sequence \eqref{seqSerreA} tensored by $\cU(-1)$ gives
\begin{gather*}
\dim\Ext^2_{\p3}\big(\cU,\cU\big)=h^2\big(\p3,\cU\otimes\cU\big)\le h^2\big(\p3,\cU(-1)\big)+h^2\big(\p3,\cU\otimes \cI_{A\vert\p3}(1)\big).
\end{gather*}

It is easy to deduce from Equality \eqref{Normal} and the cohomology of Sequences \eqref{seqSerreA} and \eqref{seqStandard} that all the dimension on the left are actually zero, thus Equality \eqref{RRgeneral} for $\cU$, yields that the component of the moduli space $\cM_{\p3}(2;0,\alpha)$ containing $\cU$ is generically smooth of dimension $\dim\Ext^1_{\p3}\big(\cU,\cU\big)=8\alpha-3$.

Recently, in \cite[Theorems 1.1]{Tik1,Tik2}, the author proved that $\cI(\alpha)$ is irreducible: the computations above also show that its dimension is $\dim(\cI(\alpha))=8\alpha-3$. Moreover, in \cite[Theorem 8.6]{J--V}, the authors succeeded in proving the smoothness of $\cI(\alpha)$ as a by--product of a more general result. Such a result and Lemma \ref{lExt12} then immediately imply that $\Ext^2_{\p3}\big(\cU,\cU\big)=0$ for each instanton bundle with charge $\alpha$ on $\p3$.
\end{example}

\begin{example}
\label{eFanoCyclic}
In \cite{Fa} the author defined instanton bundles on a Fano threefold $X$ with $\Pic(X)$ generated by $\cO_X(h)$, as the $\mu$--stable rank $2$ bundles $\cE$ such that $c_1(\cE)=(2q_X-i_X)h$ and $h^1\big(X,\cE(-q_Xh)\big)=0$: a similar definition can be found in in \cite{Kuz} when $i_X=2$. The description of instanton bundles on such Fano threefolds is the main object of the paper \cite{Fa}.

The aforementioned theorem of Maruyama still implies that $\cE$ is earnest in this case. As we mentioned in the introduction, the problem of showing that an instanton bundle is generically trivial on $\Lambda$ is still open (see \cite[Section 3.7 and Conjecture 3.16]{Kuz}). \end{example}

\begin{example}
\label{eFanoSix}
In \cite{M--M--PL} instanton bundles on a Fano variety with Picard number $\varrho_X\ge2$ were studied for the first time. More precisely the authors deal with instanton bundles on the flag variety, i.e. the general hyperplane section $X$ of the Segre embedding $\p2\times\p2\subseteq\p8$. The two projections $\p2\times\p2\to\p2$ induce projections $\pi_i\colon X\to \p2$ and $\Pic(X)$ is freely generated by  $\cO_X(h_i)\cong\pi^*_i\cO_{\p2}(1)$: moreover $A^2(X)$ is freely generated by $h_1^2$ and $h_2^2$ with the relation $h_1h_2=h_1^2+h_2^2$. Notice that lines on $X$ are exactly the cycles in the classes $h_1^2$ and $h_2^2$, hence $\Lambda=\Lambda_1\cup\Lambda_2$, $\Lambda_i\cong\p2$ be the Hilbert scheme of lines in the class $h_i^2$.

The authors define instanton bundles on $X$ as the ones such that $c_2(\cE)$ is a multiple of $h_1h_2$: the latter extra assumption is motivated by reasons related to physics. Nevertheless, if we even do not consider such a restriction, most of the results proved in \cite{M--M--PL} remain valid. In particular we can still find a monad for $\cE$ as in  \cite[Theorem 4.1]{M--M--PL}. If $c_2(\cE)=k_1h_1^2+k_2h_2^2$, then it looks like
\begin{align*}
0\longrightarrow\cO_X(-h_1)^{\oplus k_1}\oplus\cO_X(-h_2)^{\oplus k_2}&\longrightarrow\cO_X^{\oplus 2k_1+2k_2+2}\longrightarrow\\
&\longrightarrow \cO_X(h_1)^{\oplus k_1}\oplus\cO_X(h_2)^{\oplus k_2}\longrightarrow0.
\end{align*}
The above monad induces the two short exact sequences
\begin{equation}
\label{DisplayFlag}
\begin{gathered}
0\longrightarrow \cK\longrightarrow\cO_X^{\oplus 2k_1+2k_2+2}\longrightarrow \cO_X(h_1)^{\oplus k_1}\oplus\cO_X(h_2)^{\oplus k_2}\longrightarrow0,\\
0\longrightarrow \cO_X(-h_1)^{\oplus k_1}\oplus\cO_X(-h_2)^{\oplus k_2}\longrightarrow \cK\longrightarrow\cE\longrightarrow0.
\end{gathered}
\end{equation}

Recall that if $\alpha_1,\alpha_2\in\bZ$,  with $\alpha_1\le \alpha_2$, we have 
$$
h^i\big(X,\cO_X(\alpha_1h_1+\alpha_2h_2)\big)\ne0
$$
if and only if
\begin{itemize}
\item $i=0$ and $\alpha_1\ge0$;
\item $i=1$ and $\alpha_1\le -2$, $\alpha_1+\alpha_2+1\ge0$;
\item $i=2$ and $\alpha_2\ge0$, $\alpha_1+\alpha_2+3\le0$;
\item $i=3$ and $\alpha_2\le -2$.
\end{itemize}
In all these cases
\begin{equation}
\label{Flag}
h^i\big(X,\cO_X(\alpha_1h_1+\alpha_2h_2)\big)=(-1)^i\frac{(\alpha_1+1)(\alpha_2+1)(\alpha_1+\alpha_2+2)}{2}.
\end{equation}
(see \cite[Proposition 2.5]{C--F--M2}).

If $\cO_X(D)\cong \cO_X(a_1h_1+a_2h_2)$ has sections, then the numbers $a_i$'s must be non--negative, thanks to Equality \eqref{Flag}. It is immediate to check that 
\begin{gather*}
h^2\big(X,\cO_X(-(a_1+2)h_1-(a_2+1)h_2)\big)=h^2\big(X,\cO_X(-(a_1+1)h_1-(a_2+2)h_2)\big)=0,\\
h^1\big(X,\cO_X(-(a_1+1)h_1-(a_2+1)h_2)\big)=0,\\
h^0\big(X,\cO_X(-a_1h_1-(a_2+1)h_2)\big)=h^0\big(X,\cO_X(-(a_1+1)h_1-a_2h_2)\big)=0.
\end{gather*}
Thus the cohomology of Sequences \eqref{DisplayFlag} tensored by $\cO_X(-h-D)$ yields that $\cE$ is earnest, because $a_1,a_2\ge0$.

Notice that in this case not every instanton bundle is automatically $\mu$--stable. Indeed, in \cite{M--M--PL} the authors show the existence of strictly $\mu$--semistable instanton bundles proving that they are exactly the ones fitting in the extensions
$$
0\longrightarrow\cO_X(D)\longrightarrow\cE\longrightarrow\cO_X(-D)\longrightarrow0,
$$
 where $\cO_X(D)\cong\cO_X(dh_1-dh_2)$ for some non--zero integer $d$ (in particular see \cite[Proposition 2.5]{M--M--PL}). It can be easily proved that such strictly $\mu$--semistable instanton bundles are always  generically trivial on exactly one component of $\Lambda$.

The study of instanton bundles on the Segre embedding $\p1\times\p1\times\p1\subseteq\p7$ leads to similar results: the interested reader can find them in \cite{An}.
\end{example}

\begin{remark}
\label{rVeronese}
Let $X$ be a Fano threefold endowed with an ample line bundle $\cO_F(H)$ such that $\omega_X\cong\cO_X(-i_{X,H}H)$ for some integer $i_{X,H}$ (if any).

In principle, one could introduce the notion of instanton bundle with respect to $\cO_F(H)$ by replacing the index $i_X$ with  $i_{X,H}$  in Definition \ref{dInstanton} and defining
$$
q_{X,H}:=\left[\frac{i_{X,H}}2\right].
$$
As an example we deal with $\p3$.

The $2$--uple embedding $X\subseteq\p9$ of $\p3$ is a del Pezzo threefold. In particular $\cO_{X}(H):=\cO_{\p3}(2)$ and $\omega_X\cong\cO_{\p3}(-4)\cong\cO_X(-2H)$ so that $i_{X,H}=2$ and $q_{X,H}=1$: notice that in this case though $\Pic(X)$ is cyclic, it is not generated by $\cO_X(H)$. Since $i_{X,H}$ is even, it follows that an instanton bundle $\cE$  with respect to $\cO_F(H)$ still satisfies $c_1(\cE)=0$.

We already know that a vector bundle $\cE$ is $\mu$--stable with respect to $\cO_{\p3}(1)$ if and only if the same is true with respect to $\cO_{X}(H)$. Moreover, $h^1\big(X,\cO_{X}(-H)\big)=h^1\big(\p3,\cO_{\p3}(-2)\big)$. It follows that $\cE$ is an instanton bundle on $X$ with respect to $\cO_X(H)$, if and only if it is an instanton bundle on $\p3$ with respect to $\cO_{\p3}(1)$. 

Notice that $c_2(\cE)H=2c_2(\cE)c_1(\cO_{\p3}(1))$, hence instanton bundles of minimal charge coincide. Finally $X$ does not contain lines, hence every instanton bundle on $X$ is generically trivial.

Similarly, we can also consider the $4$--uple embedding $X\subseteq\p{34}$ of $\p3$. In particular $\cO_{X}(H):=\cO_{\p3}(4)$ and $\omega_X\cong\cO_{\p3}(-4)\cong\cO_X(-H)$, hence $i_{X,H}=1$ and $q_{X,H}=0$. Every instanton bundle $\cE$  with respect to $\cO_F(H)$ must satisfy $c_1(\cE)=-H$ and $h^1\big(X,\cE\big)=0$. Thus the bundle $\cF:=\cE(2)$ satisfies $c_1(\cF)=0$ and $h^1\big(X,\cF(-2)\big)=0$, hence it is either $\mu$--stable or $\cF\cong\cO_X^{\oplus2}$, thanks to \cite[Remark in Section II.3.4]{O--S--S}. We conclude that $\cE(2)$,  if indecomposable, is an instanton on $\p3$ with respect to $\cO_{\p3}(1)$.
\end{remark}

\section{General results on instanton bundles}
\label{sGeneral}
In this section we list some results on instanton bundles which holds true on every Fano threefold.

Let $X$ be a variety with a fixed polarization $\cO_X(h)$:
a sheaf $\cF$ on $X$ has {\sl natural cohomology} in degree $\lambda\in \bZ$ (with respect to $\cO_X(h)$), if $h^i\big(X,\cF(\lambda h)\big)\ne0$ for at most one $i$. 
In \cite{O--S} the authors gave a different definition of instanton sheaves in terms of sheaves with natural cohomology in a certain range. We can state the following proposition.

\begin{proposition}
\label{pNatural}
Every $\mu$--semistable bundle $\cE$ of rank $2$  with natural cohomology in degree $-q_X$ on a Fano threefold  $X$, such that $c_1(\cE)=(2q_X-i_X)h$, and $h^0\big(X,\cE\big)=0$ is an instanton bundle.

Every instanton bundle $\cE$ on a Fano threefold $X$ has natural cohomology in degree $\lambda$ in the range $-2q_X\le\lambda \le 0$. More precisely, $h^i\big(X,\cE(\lambda h)\big)=0$ unless either $1-q_X\le\lambda\le0$ and $i=1$, or $-2q_X\le\lambda\le-1-q_X$ and $i=2$.
\end{proposition}
\begin{proof}
Equality \eqref{RRgeneral} implies $\chi(\cE(-q_Xh))=0$, hence 
\begin{equation}
\label{VanishingQ}
h^i\big(S,\cE(-q_Xh)\big)=0,\qquad i\ge0,
\end{equation}
because the cohomology of $\cE$ is natural in degree $-q_X$. Thus $\cE$ is an instanton bundle by definition. 

Conversely, let us first examine the case $\lambda=-q_X$. Since $h^0\big(X,\cE\big)=0$, it follows that $h^0\big(X,\cE(-q_Xh)\big)=h^3\big(X,\cE(-q_Xh)\big)=0$. Moreover, thanks to Equality \eqref{RRgeneral} we have $\chi(\cE(-q_Xh))=0$ because $c_1(\cE)=(2q_X-i_X)h$, hence
$$
h^1\big(X,\cE(-q_Xh)\big)=h^2\big(X,\cE(-q_Xh)\big).
$$
Equality  \eqref{VanishingQ} then follows from the instantonic condition.

Let us now examine the case $\lambda\ne -q_X$. We have
$$
h^i\big(X,\cE(\lambda h)\big)=h^{3-i}\big(X,\cE((-\lambda-2q_X) h)\big)=0, 
$$
thanks to Equality \ref{Serre}, hence it suffices to prove the thesis for $1-q_X\le \lambda\le0$: in particular we can assume $i_X\ge2$ from now on. 

In the cases $i_X=2,3$, it suffices to check that
\begin{equation}
\label{VanishingH^2}
h^2\big(X,\cE\big)=0,
\end{equation}
because $q_X=1$ in these cases. Let $H$ be a general hyperplane section of $X$: $H$ is then a Fano surface and we have the exact sequence
\begin{equation*}
\label{seqSection}
0\longrightarrow\cO_X(-h)\longrightarrow\cO_X\longrightarrow\cO_H\longrightarrow0.
\end{equation*}
Since $h^0\big(X,\cE\big)=h^1\big(X,\cE(-h)\big)=0$, it follows that $h^0\big(H,\cE\otimes\cO_H\big)=0$: Equality \eqref{Serre} then yields $h^2\big(H,\cE\otimes\cO_H\big)=0$ too. The cohomology of the above sequence and Equality \eqref{VanishingQ} yield the Equality \eqref{VanishingH^2}, hence the statement is proved when $i_X=2,3$.

In the case $i_X=4$, it is easy to check that $\cE$ has natural cohomology looking at the exact sequences associated to the monad defining $\cE$ (see \cite[Example 5 of Section II.3.2]{O--S--S}). 
\end{proof}

\begin{corollary}
\label{cCharge}
If $\cE$ is an instanton bundle on a Fano threefold $X$, then 
\begin{equation}
\label{Minimal}
c_2(\cE) h\ge\left\lbrace\begin{array}{ll} 
1\quad&\text{if $i_X=4$,}\\
2\quad&\text{if $i_X=2,3$,}\\
\frac{\deg(X)}4\quad&\text{if $i_X=1$.}
\end{array}\right.
\end{equation}
\end{corollary}
\begin{proof}
If $i_X\ge2$, then $\chi(\cE)=-h^1\big(X,\cE\big)$ thanks to Proposition \ref{pNatural}. Equality \eqref{RRgeneral} implies
\begin{equation}
\label{ChiE}
\chi(\cE)=\left\lbrace\begin{array}{ll} 
2-2c_2(\cE) h\quad&\text{if $i_X=4$,}\\
2-c_2(\cE) h\quad&\text{if $i_X=2$.}
\end{array}\right.
\end{equation}
By combining the above equalities with the  non--positivity of $\chi(\cE)$, we deduce the statement for $i_X=2,4$.

If $i_X=3$ see \cite[proof of Lemma 2]{Fa}. Finally, let us consider the cases $i_X=1$. If $H$ is a general hyperplane section of $X$, then we know that $\cE\otimes\cO_H$ is $\mu$--semistable (see \cite[Theorem 3.1]{Ma}), hence the Bogomolov inequality for $\cE\otimes\cO_H$ yields the last inequality of the statement.
\end{proof}

\begin{remark}
\label{rMinimal}
An instanton bundle on $X$ is called {\sl minimal} if equality holds in Inequality \eqref{Minimal}.

When $i_X=4$, every minimal instanton bundle $\cE$ is a null--correlation bundle, i.e. it fits into an exact sequence of the form
$$
0\longrightarrow\cO_{\p3}(-1)\longrightarrow\Omega_{\p3}^1(1)\longrightarrow\cE\longrightarrow0
$$
(see \cite[Lemma II.4.3.2]{O--S--S}), hence $h^1\big(\p3,\cE(-1)\big)=1$.

If $i_X=2,3$ and $\cE$ is a minimal instanton bundle, then Equality \eqref{ChiE} implies that $h^i\big(X,\cE(\lambda)\big)=0$ in the range $-2\le \lambda\le0$ and $i\ge0$. It follows that $\cE(h)$ is an Ulrich bundle with $c_1(\cE(h))=(2q_X-i_X+2)h$ in the sense of \cite{E--S--W}. Ulrich bundles of rank $2$ on Fano threefolds with $i_X=2,3$ are completely classified. The unique Ulrich bundle of rank $2$ on the smooth quadric is $\cS(h)$ where $\cS$ is the spinor bundle (see \cite{E--He}). If $i_X=2$, then Ulrich bundles of rank $2$ have been described in \cite{A--C, C--F--M1, C--F--M2, C--F--M3, Cs1}. 

When $i_X=1$ minimal instanton bundles are not necessarily Ulrich, because we have too few vanishings in cohomology. 
\end{remark}

The last result of this section deals with Fano threefolds containing a plane $E$ (see \cite{Tsu} for their classification).

\begin{lemma}
\label{lEarnest}
Let $\cE$ be a instanton on a Fano threefold $X$ with $i_X\le2$ containing a plane $E$.

The following assertions are equivalent:
\begin{enumerate}
\item $\cE\otimes\cO_E$ is $\mu$--semistable;
\item $h^1\big(X,\cE(-E-q_Xh)\big)=0$.
\end{enumerate}
The above assertions imply the following one:
\begin{enumerate}
\item[(3)] $\cE$ is generically trivial on $\Lambda_E$.
\end{enumerate}
\end{lemma}
\begin{proof}
On the one hand, the cohomology of Sequence \eqref{seqRestriction} with $D:=E$ tensored by $\cE(-q_Xh)$ and the instantonic condition imply that $\cE\otimes\cO_E$ is $\mu$--semistable if and only if $h^1\big(X,\cE(-E-q_Xh)\big)=0$ thanks to \cite[Lemma II.1.2.5]{O--S--S}. 

Notice that if $\cE\otimes\cO_E$ is $\mu$--semistable, then it splits on the general line $L\subseteq E$ as either $\cO_{\p1}^{\oplus2}$, if $i_X$ is even, or $\cO_{\p1}(-1)\oplus\cO_{\p1}$, if $i_X$ is odd thanks to \cite[Corollary 2 of Theorem II.2.1.4]{O--S--S}. Trivially, the locus of lines $L\in\Lambda_E$ with such a splitting property is open by semicontinuity, hence $\cE$ is generically trivial on $\Lambda_E$.
\end{proof}

\section{The blow up of $\p3$ at a point}
\label{sBlow}
From now on we will focus our attention of the blow up $F:=\Bl_P\p3$ of $\p3$ at a point $P$. The threefold $F$ is endowed with two natural morphisms, the blow up map $\sigma\colon F\to \p3$ and the projection $\pi\colon F\cong\bP(\cP)\to\p2$, where $\cP:=\cO_{\p2}\oplus\cO_{\p2}(1)$. Recall that $E:=\sigma^{-1}(P)\cong\p2$ and $\sigma$ induces an isomorphism $F\setminus \sigma^{-1}(P)\cong\p3\setminus\{\ P\ \}$.

As explained in the introduction, we have the classes $\xi$ and $f$ of $\cO_{\bP(\cP)}(1)$ and $\pi^*\cO_{\p2}(1)$ respectively.  Trivially $\pi^*\cO_{\p2}(1)$ is globally generated. Since $\cP$ is globally generated, the same holds for $\cO_F(\xi)\cong\cO_{\bP(\cP)}(1)$: moreover, $\cO_F(\xi)\cong \sigma^*\cO_{\p3}(1)$. The embedding $F\subseteq\p8$ is induced by the linear system $\cO_F(h)=\cO_F(\xi+f)$ and $\omega_F\cong\cO_F(-2h)$: in particular $i_F=2$. 

Recall that  $\xi^3=\xi^2 f=\xi f^2$ are the class of a point: moreover, $\Omega=c_2(\Omega_F^1)=6\xi f$. In particular, if $\cE$ is an instanton bundle with charge $\alpha\xi^2+\beta f^2$ on $F$, then Equalities \eqref{RRgeneral} and $2hc_2(\Omega_{X})={24}$ yield
\begin{equation}
\label{RRBlowUp}
\begin{aligned}
\chi(\cE(a\xi+b f))&=\frac{a^3}3+a^2b+ab^2+2a^2+b^2+4ab+\\
&+3b+\frac{11a}3+2-(a+b+2)\alpha-(a+1)\beta.
\end{aligned}
\end{equation}

Let $H\subseteq\p3$ be a plane through $P$. On the one hand, $\sigma^{-1}(H)$ is in the class of $\xi$. On the other hand, $\sigma^{-1}(H)$ is the union of $E$ with the strict transform of $H$. Such a strict transform is in the linear system $\vert f\vert$, hence $E$ is the unique element in $\vert \xi-f\vert$. Notice that $Eh^2=1$, hence $E\subseteq \p8$ is a plane.

Conversely, let $H\subseteq F\subseteq\p8$ be a plane and assume $\cO_F(\alpha\xi+\beta f)\cong\cO_F(H)$, so that $4\alpha+3\beta=Hh^2=1$. Since both $\cO_F(f)$ and $\cO_F(\xi)$ are globally generated, it follows that $\alpha+\beta=H\xi^2$ and $\alpha=Hf^2$ are both non--negative. We deduce that $\alpha=-\beta=1$ necessarily, i.e. $E$ is the only plane contained in $F$. 

In order to prove Theorem \ref{tSimplify}  stated in the introduction we will make use the following result.

\begin{lemma}
\label{lAO}
Let $\cE$ be an instanton bundle on $F$.

Then $\cE$  is the cohomology in degree $0$ of a complex $\cC^\bullet$ where
$$
\cC^i:= \bigoplus_{q-s=i}\bigoplus_{a+b=s} H^q\big(F,\cE(-a\xi-(b-a)f)\big)\otimes \pi^*\Omega_{\p2}^{b}(bf)\otimes\wedge^{a}\cO_F(-\xi).
$$
\end{lemma}
\begin{proof}
Recall that $F\cong\bP(\cP)$, hence we can apply \cite[Theorem 8]{A--O}: notice that, with the notation in that paper, $\bP(\cG):=\bP(\mathrm{Sym}(\cG^\vee))$ for each locally free sheaf $\cG$. In particular, in our case, we thus have $\cG=\cP^\vee\cong\cO_{\p2}\oplus\cO_{\p2}(-1)$: in order to apply \cite[Theorem 8]{A--O} we must consider $\cG(1)$, hence the relative universal line bundle therein is $\cO_F(\xi-f)$.

 The universal quotient bundle $\cQ$ on $F$ can be computed by dualizing the following relative Euler exact sequence 
\begin{equation}
\label{seqRelative}
0\longrightarrow\cO_F(f-2\xi)\longrightarrow\cO_F(-\xi)\oplus\cO_F(f-\xi)\longrightarrow\cO_F\longrightarrow0,
\end{equation}
tensored by the the relative universal line bundle, i.e. $\cQ\cong\cO_F(\xi)$.
\end{proof}

We deduce from the above statement that in order to prove Theorem \ref{tSimplify} we have to compute the cohomologies $h^q\big(F,\cE(-a\xi-(b-a)f)\big)$ for certain values of $q,a,b$. Trivially we must have $0\le q\le 3$, $0\le b\le 2$ and $0\le a\le 1$. Thus we have to compute $h^q\big(F,\cE\otimes\cF_p\big)$ for $0\le q\le3$ and $0\le p\le 5$ where
\begin{align*}
(\cF_0,\cF_1,\cF_2,\cF_3,\cF_4,\cF_5):&=(\cO_F(-\xi-f),\cO_F(-\xi),\\
&\cO_F(-\xi+f),\cO_F(-2f),\cO_F(-f),\cO_F)
\end{align*}
is the Orlov collection (with respect to the polarization $\cO_F(\xi-f)$: see \cite[Corollary 2.6]{Orl}).

\begin{proposition}
\label{pTable}
Let $\cE$ be an instanton bundle on $F$.

If $c_2(\cE)=\alpha\xi^2+\beta f^2$ and $\gamma:=h^1\big(F,\cE(-2\xi)\big)$,
then $h^q\big(F,\cE\otimes\cF_p\big)$ is the number in position $(p,q)$ in the following table.
\begin{table}[H]
\centering
\bgroup
\def\arraystretch{1.5}
\begin{tabular}{ccccccc}
\cline{1-6}
\multicolumn{1}{|c|}{0} & \multicolumn{1}{c|}{0} & \multicolumn{1}{c|}{0} & \multicolumn{1}{c|}{0} & \multicolumn{1}{c|}{0} & \multicolumn{1}{c|}{0} & $q=3$ \\ \cline{1-6}
\multicolumn{1}{|c|}{0} & \multicolumn{1}{c|}{0} & \multicolumn{1}{c|}{0} & \multicolumn{1}{c|}{$\gamma$} & \multicolumn{1}{c|}{0} & \multicolumn{1}{c|}{0} & $q=2$ \\ \cline{1-6}
\multicolumn{1}{|c|}{0} & \multicolumn{1}{c|}{$\alpha$} & \multicolumn{1}{c|}{$2\alpha$} & \multicolumn{1}{c|}{$\beta+\gamma$} & \multicolumn{1}{c|}{$\alpha+\beta$} & \multicolumn{1}{c|}{$2\alpha+\beta-2$}& $q=1$ \\ \cline{1-6}
\multicolumn{1}{|c|}{0} & \multicolumn{1}{c|}{0} & \multicolumn{1}{c|}{0} & \multicolumn{1}{c|}{0} & \multicolumn{1}{c|}{0} & \multicolumn{1}{c|}{0} & $q=0$ \\ \cline{1-6}
$p=0$ & $p=1$ & $p=2$ & $p=3$ & $p=4$ & $p=5$
\end{tabular}
\egroup
\caption{The values of $h^q\big(F,\cE\otimes\cF_p\big)$}
\end{table}
\end{proposition}
\begin{proof}
Some of the vanishing above are trivial. Indeed $h^0\big(F,\cE\otimes\cF_5\big)=0$ by definition. It follows that $h^0\big(F,\cE\otimes\cF_p\big)$ for each $p$, because $\cE\otimes\cF_p$ is a subbundle of $\cE$. The same argument and Equality \eqref{Serre} yield $h^3\big(F,\cE\otimes\cF_p\big)$ for each $p$. 

Proposition \ref{pNatural} implies that $h^q\big(F,\cE\otimes\cF_0\big)$ for each $q$ and that there is at most one $q$ such that $h^q\big(F,\cE\otimes\cF_5\big)\ne0$. Equality \eqref{RRBlowUp} implies that $\chi(\cE)=2-\beta-2\alpha$ which is non--positive thanks to Corollary \ref{cCharge}, hence $h^1\big(F,\cE\otimes\cF_5\big)=2\alpha+\beta-2$.

Equality \eqref{Serre} implies $h^2\big(F,\cE\otimes\cF_3\big)=h^1\big(F,\cE(-2\xi)\big)=\gamma$, hence Equality \eqref{RRBlowUp} returns $h^1\big(F,\cE\otimes\cF_3\big)=\beta+\gamma$.

Notice that we have the two exact sequences 
\begin{gather}
\label{seqOmega}
0\longrightarrow\pi^*\Omega^1_{\p2}\longrightarrow\cO_F(-f)^{\oplus3}\longrightarrow\cO_F\longrightarrow0,\\
\label{seqOmegaDual}
0\longrightarrow\cO_F(-3f)\longrightarrow\cO_F(-2f)^{\oplus3}\longrightarrow\pi^*\Omega^1_{\p2}\longrightarrow0.
\end{gather}
The first exact sequence is the pull--back of the Euler exact sequence on $\p2$ via $\pi$. The second one is the pull--back of the twisted dual of the same sequence. 

The cohomology of Sequence \eqref{seqOmega} tensored by $\cE(f-\xi)$ and the vanishing $h^0\big(F,\cE\otimes\cF_1\big)=0$ imply $h^0\big(F,\cE\otimes\pi^*\Omega^1_{\p2}(f-\xi)\big)=0$. Such a vanishing, the cohomology of Sequence \eqref{seqOmegaDual} tensored by $\cE(f-\xi)$ and $h^1\big(F,\cE\otimes\cF_0\big)=0$ return $h^1\big(F,\cE(-\xi-2f)\big)=0$. Thanks to Equality \eqref{Serre}, the latter vanishing implies $h^2\big(F,\cE\otimes\cF_1\big)=0$. We deduce $h^2\big(F,\cE\otimes\cF_1\big)=\alpha$ thanks to Equality \eqref{RRBlowUp}.

The  vanishings $h^2\big(F,\cE\otimes\cF_0\big)=h^2\big(F,\cE\otimes\cF_1\big)=0$ and the cohomology of Sequence \eqref{seqRelative} tensored by $\cE(-f)$ yield the inequality
$$
h^2\big(F,\cE\otimes\cF_4\big) \le h^3\big(F,\cE(-2\xi)\big)=h^0\big(F,\cE(-2f)\big)\le h^0\big(F,\cE\big)=0.
$$
Thus $h^1\big(F,\cE\otimes\cF_4\big)=\alpha+\beta$ (see Equality \eqref{RRBlowUp}). 
 
Let us finally consider the cohomology of Sequence \eqref{seqOmegaDual} tensored by $\cE(-\xi)$. Since
$$
h^0\big(F,\cE\otimes\pi^*\Omega^1_{\p2}(-\xi)\big)\le h^0\big(F,\cE\otimes\pi^*\Omega^1_{\p2}(f-\xi)\big)=0,
$$
and $h^1\big(F,\cE(-\xi-2f)\big)=0$, it follows from Equality \eqref{Serre} that
$$
h^2\big(F,\cE\otimes\cF_{2}\big)\le h^3\big(F,\cE(-\xi-3f)\big)=h^0\big(F,\cE(-\xi+f)\big)=0,
$$
hence $h^1\big(F,\cE\otimes\cF_{2}\big)=2\alpha$, thanks again to Equality \eqref{RRBlowUp}.

The statement is then completely proved.
\end{proof}
   
The next corollary is an immediate by--product of Proposition \ref{pTable}. It deals with the numbers $\alpha$, $\beta$, $\gamma$ clarifying their roles and specializing Corollary \ref{cCharge} to the case of instantons on $F$.

\begin{corollary}
\label{cChargeBlow}
Let $\cE$ be an instanton with charge $\alpha\xi^2+\beta f^2$ on $F$.

Then $\alpha\ge0$, $\alpha+\beta\ge0$, $\beta+\gamma\ge0$ and $2\alpha+\beta\ge2$. 
\end{corollary}
\begin{proof}
From Table 1 we deduce that the number listed in the statement are dimensions of certain vector spaces, thus they must be non--negative.
\end{proof}

In the proof of Theorem \ref{tSimplify} we will need the cohomology of certain twist of $\pi^*\Omega_{\p2}^b$. We compute it in the next propositions: we first compute the cohomology of such twists when $b=0,2$ (see also \cite[Proposition 2.3]{C--F--M2} where the statement though equivalent, is slightly different). 

\begin{proposition}
\label{pLineBundle}
We have
$$
\begin{aligned}
h^0\big(F,\cO_F(a\xi+b f)\big)&=\sum_{j=0}^{a}{b +2+j\choose2},\\
h^1\big(F,\cO_F(a\xi+b f)\big)&=\sum_{j=0}^{-a-2}{b +1-j\choose2},\\
h^2\big(F,\cO_F(a\xi+b f)\big)&=\sum_{j=0}^{a}{-b -1-j\choose2},\\
h^3\big(F,\cO_F(a\xi+b f)\big)&=\sum_{j=0}^{-a-2}{-b +j\choose2}
\end{aligned}
$$
where the summation is $0$ if the upper limit is smaller than the lower limit.
\end{proposition}
\begin{proof}
On the one hand, if $a\ge-1$, then \cite[Exercises III.8.1, III.8.3 and III.8.4]{Ha2} implies that
$$
h^i\big(F,\cO_F(a\xi+b f)\big)=h^i\big(\p2,\cO_{\p2}(b )\otimes\pi_*\cO_F(a\xi)\big)=\sum_{j=0}^{a}h^i\big(\p2,\cO_{\p2}(b +j)\big).
$$
On the other hand, if $a\le-1$, then Equality \eqref{Serre} yields
$$
h^i\big(F,\cO_F(a\xi+b f)\big)=h^{3-i}\big(F,\cO_F(-(a+2)\xi-(b +2)f)\big).
$$
The statement then follows by combining the above equalities.
\end{proof}

A trivial consequence of the above proposition is that $\cO_F(a\xi+b f)$ is an effective line bundle if and only if $a,a+b\ge0$. 

\begin{corollary}
\label{cPositive}
The cycle $\alpha\xi^2+\beta f^2\in A^2(F)$ is in $\Mov(F)$ if and only if $\alpha,\beta\ge0$. 
\end{corollary}
\begin{proof}
Since the pseudo--effective cone is generated by the effective divisor (i.e. divisors in linear systems  $\vert a\xi+ bf\vert$ with $a,a+b\ge0$), it suffices to check that 
$$
(\alpha\xi^2+\beta f^2)(a\xi+bf)=\alpha(a+b)+\beta a
$$
for each $a,b$ such that $a,a+b\ge0$ if and only if $\alpha,\beta\ge0$, which is trivial.
\end{proof}

Now we deal with the case $b=1$.

\begin{proposition}
\label{pOmega}
We have
$$
\begin{aligned}
h^0\big(F,\pi^*\Omega^1_{\p2}(a\xi+b f)\big)&=\sum_{j=0}^{a}{{b +1+j}\choose1}{{b -1+j}\choose1},\\
h^1\big(F,\pi^*\Omega^1_{\p2}(a\xi+b f)\big)&=\left\lbrace\begin{array}{ll} 
1\quad&\text{if $a\ge -b\ge0$,}\\
\sum_{j=0}^{-a-2}{{b -j}\choose1}{{b -2-j}\choose1}\quad&\text{elsewhere,}
\end{array}\right.\\
h^2\big(F,\pi^*\Omega^1_{\p2}(a\xi+b f)\big)&=\left\lbrace\begin{array}{ll} 
1\quad&\text{if $a\le-b-1\le-2$,}\\
\sum_{j=0}^{a}{{-b +1-j}\choose1}{{-b -1-j}\choose1}\quad&\text{elsewhere,}
\end{array}\right.\\
h^3\big(F,\pi^*\Omega^1_{\p2}(a\xi+b f)\big)&=\sum_{j=0}^{-a-2}{{-b+2+j}\choose1}{{-b+j}\choose1}
\end{aligned}
$$
where the summation is $0$ if the upper limit is smaller than the lower limit.
\end{proposition}
\begin{proof}
It suffices to use the Bott formulas (see \cite[Formulas at p.8]{O--S--S}) and the same argument already used in the proof of Proposition \ref{pLineBundle}.
\end{proof}

The following remark will be helpful for proving Theorem \ref{tSimplify} stated in the introduction.

\begin{remark}
\label{rSmooth}
We show that $\vert a\xi+bf\vert$ contains a smooth integral divisor $D$ if and only if either $a,b\ge0$, or $a=-b=1$.

To this purpose we first notice that $\cO_F(a\xi+b f)$ is globally generated if and only if $a,b\ge0$. Indeed, on the one hand, if $a,b\ge0$ the assertion is a trivial consequence of the existence of a surjective morphism $\pi^*\cP\to\cO_F(\xi)$. On the other hand, if $\cO_F(a\xi+b f)$ is globally generated, then $a=(a\xi+b f)f^2$ and $b=(a\xi+b f)(\xi^2-f^2)$ must be non--negative.

If $a=-b=1$, then $D=E$ which is trivially smooth and integral. If $a,b\ge0$, then $\cO_F(a\xi+b f)$ is globally generated, hence $\vert a\xi+bf\vert$ contains a smooth integral divisor thanks to the Bertini theorem. 

Conversely, assume that $\vert a\xi+bf\vert$ contains a smooth integral divisor. Thus if $\cO_F(a\xi+b f)$ is not globally generated, then $a\ge1$ and $-1\ge b\ge-a$, thanks to Proposition \ref{pLineBundle}.
If $E\not\subseteq D$, then $D\cap E$ is necessarily a curve or it is empty, hence we should have $0\le DEh=b\le -1$, a contradiction. Thus $D=E+R$ for some divisor $R$ on $X$: since $D$ is smooth and integral, it follows that $R=0$, i.e. $D=E$.
\end{remark}

We are now ready to prove Theorem \ref{tSimplify} as an almost immediate consequence of Lemma \ref{lAO} and Proposition \ref{pTable}.
\medbreak
\noindent{\it Proof of Theorem \ref{tSimplify}.}
By applying Lemma \ref{lAO} using the values $h^q\big(F,\cE\otimes\cF_p)$ calculated in Proposition \ref{pTable} we obtain the complex $\cC^\bullet$ where
\begin{gather*}
\cC^{-1}:=\cO_F(-f)^{\oplus\beta+\gamma}\oplus\pi^*\Omega_{F}^1(f-\xi)^{\oplus\alpha},\\
\cC^0:=\cO_F(-f)^{\oplus\gamma}\oplus\pi^*\Omega_{\p2}^1(f)^{\oplus\alpha+\beta}\oplus\cO_F(-\xi)^{\oplus2\alpha},\\
\cC^1:=\cO_F^{\oplus2\alpha+\beta-2},
\end{gather*}

We have the two short exact sequences
\begin{equation}
\label{Display}
\begin{gathered}
0\longrightarrow \cK\longrightarrow \cC^0\longrightarrow\cC^{1}\longrightarrow0,\\
0\longrightarrow \cC^{-1}\longrightarrow \cK\longrightarrow\cE\longrightarrow0.
\end{gathered}
\end{equation}
Thus
\begin{equation}
\label{BoundDisplay}
h^i\big(F,\cE\otimes\mathcal L\big)\le \sum_{j=-1}^1h^{i-j}\big(F,\cC^{j}\otimes\mathcal L\big)
\end{equation}
for each $\mathcal L\in\Pic(F)$.

Taking $\mathcal L:=\cO_F(-2\xi)$, 
Propositions \ref{pLineBundle}, \ref{pOmega} and Inequality \eqref{BoundDisplay} imply
\begin{align*}
\gamma=h^1\big(F,\cE(-2\xi)\big)\le \alpha h^2\big(F,\pi^*\Omega_{F}^1(-3\xi+f)\big)=\alpha.
\end{align*}

Conversely, if $\cE$ is the cohomology $\cC^\bullet$, we have Sequences \eqref{Display}. Easy and tedious computations lead to the equalities
\begin{gather*}
c_1(\cE)=c_1(\cC^0)-c_1(\cC^1)-c_1(\cC^{-1})=0,\\
c_2(\cE)=c_2(\cC^0)-c_2(\cC^1)-c_2(\cC^{-1})-c_1(\cC^0)c_1(\cC^{1})+c_1(\cC^{1})^2=\alpha\xi^2+\beta f^2.
\end{gather*}
Moreover, Inequality \eqref{BoundDisplay} trivially still holds. 

Propositions \ref{pLineBundle}, \ref{pOmega} and Inequality \eqref{BoundDisplay} with $\mathcal L:=\cO_F$ imply $h^0\big(F,\cE\big)=0$. Using Equality \eqref{Serre}, one immediately checks that
$$
h^1\big(F,\cE(-h-E)\big)=h^1\big(F,\cE(-2\xi)\big)=h^2\big(F,\cE(-2f)\big).
$$
Argueing as in the previous case with $\mathcal L:=\cO_F(-2f)$ we deduce assertion (1) of the statement. 

Now let $D$ be either $0$, or any smooth element in $\vert a\xi+bf\vert$, $D\ne E$: thanks to Remark \ref{rSmooth} we then know that $a,b\ge0$. Again the same argument with 
$$
\mathcal L:=\cO_F(-h-D)\cong\cO_F(-(a+1)-(b+1)f)
$$
yields $h^1\big(F,\cE(-h-D)\big)=0$. If $D=0$, then we deduce that $\cE$ satisfies the instantonic condition, hence is an instanton, because it is assumed $\mu$--semistable. If $D\ne0$, we obtain assertion (2).
\qed
\medbreak

\begin{remark}
If $\cE$ is assumed to be earnest, then it is the cohomology of a monad of the form
\begin{align*}
0&\longrightarrow \cO_F(-f)^{\oplus\beta}\oplus\pi^*\Omega_{F}^1(f-\xi)^{\oplus\alpha}\longrightarrow\\
&\longrightarrow \pi^*\Omega_{\p2}^1(f)^{\oplus\alpha+\beta}\oplus\cO_F(-\xi)^{\oplus2\alpha}\longrightarrow  \cO_F^{\oplus2\alpha+\beta-2}\longrightarrow 0.
\end{align*}

Starting from the monad, one cannot hope to prove that $h^1\big(F,\cE(-2\xi)\big)=\gamma$. Indeed the integer $\gamma$ can be made as large as possible without affecting the cohomology of the monad, simply by adding  automorphisms of $\cO_F(-f)^{\oplus u}$ for an arbitrarily large integer $u$.
\end{remark}

\section{General results on instanton bundles on $F$}
In this section we deal with some general properties of instanton bundles on the blow up $F$ of $\p3$ at a point.

We first describe the families of lines on $F$.

\begin{remark}
\label{rLine}
Let $L\subseteq F$ be a line. If $a\xi^2+bf^2$ is its class in $A^2(F)$ we know that
$$
2a+b=(a\xi^2+f^2)h=1,\qquad a=(a\xi^2+f^2)f, \qquad b=(a\xi^2+f^2)E.
$$
Since $\cO_F(f)$ is globally generated it follows that $a\ge0$. If $L\not\subseteq E$, then $b\ge0$, hence $a=0$, $b=1$ and the class of $L$ is $f^2$. Let $L\subseteq E$, then it is cut out on the plane $E\subseteq\p8$ by a hyperplane, hence its class is $Eh=\xi^2-f^2=Ef$.

Thus there are two families of lines on $F$ and each such line $L$ is the complete intersection of two divisors. It follows that its ideal sheaf $\cI_{L\vert F}$ fits into a Koszul--type resolution.

In the former case $L$ is a fibre of the projection $\pi\colon F\to \p2$ and we have the exact sequence
\begin{equation*}
0\longrightarrow\cO_F(-2f)\longrightarrow\cO_F(-f)^{\oplus2}\longrightarrow\cI_{L\vert F}\longrightarrow0.
\end{equation*}
The restriction of the above sequence to $L$ yields the isomorphism $\cI_{L\vert F}/\cI_{L\vert F}^2\cong\cI_{L\vert F}\otimes\cO_L\cong \cO_{\p1}^{\oplus2}$, hence $\cN_{L\vert F}\cong\cO_{\p1}^{\oplus2}$.
Let $\Lambda_\pi\subseteq\Lambda$ be the locus corresponding to these lines.

In the latter case $L$ is a line inside $E$ and we have the exact sequence 
\begin{equation*}
0\longrightarrow\cO_F(-\xi)\longrightarrow\cO_F(-f)\oplus\cO_F(-\xi+f)\longrightarrow\cI_{L\vert F}\longrightarrow0.
\end{equation*}
Argueing as above $\cN_{L\vert F}\cong\cO_{\p1}(-1)\oplus\cO_{\p1}(1)$. If $\Lambda_E$ is the locus of such lines in the Hilbert scheme $\Lambda$, then $\dim(\Lambda_E)=h^0\big(L,\cN_{L\vert F}\big)=2$ and $h^1\big(L,\cN_{L\vert F}\big)=0$. Thus $\Lambda_E$ is a smooth surface dominated by $E^\vee\cong\p2$, hence it is actually isomorphic to it by a theorem of Castelnuovo. In fact, using a similar approach, one can easily check that the Hilbert scheme of lines of any Fano threefold with $i_X\le2$ containing a plane $E$ always has $E^\vee	\cong\p2$ as a component.

It follows that $\Lambda=\Lambda_\pi\cup \Lambda_E$ where $\Lambda_\pi\cap\Lambda_E=\emptyset$ and $\Lambda_\pi\cong \Lambda_E\cong\p2$. Thus there are two ways of taking a general $L\in\Lambda$. 
\end{remark}

Thanks to the above remark, if $\cE$ is an instanton bundle which is generically trivial on $\Lambda_E$, then $\cE\otimes\cO_L\cong\cO_{\p1}^{\oplus2}$ on the general line $L\subseteq E$, hence $\cE\otimes\cO_E$ is $\mu$--semistable thanks to \cite[Lemma II.2.2.1]{O--S--S}. In particular, in this case, assertions (1), (2) and (3) in Lemma \ref{lEarnest} are actually equivalent.

Our first result is the description of minimal instanton bundles. 

\begin{proposition}
\label{pMinimal}
A bundle $\cE$ of rank $2$ on $F$ is a minimal instanton bundle if and only if $\cE(h)$ is Ulrich.

Moreover, every minimal instanton bundle  $\cE$ on $F$ is generically trivial and earnest.
\end{proposition}
\begin{proof}
In Remark \ref{rMinimal} we already checked that for each minimal instanton bundle $\cE$ on $F$, then $\cE(h)$ is Ulrich. 

Conversely, let $\cE(h)$ be an Ulrich bundle of rank $2$ on $F$. By definition we already know that $h^0\big(F,\cE\big)=h^1\big(F,\cE(-h)\big)=0$ and that $\cE$ is $\mu$--semistable (see \cite{C--G--H--S}). Thanks to the classification in \cite{C--F--M3} we know that $\cE$ is either $\pi^*\cV$, where $\cV$ is a $\mu$--stable bundle on $\p2$ with $c_1(\cV)=0$ and $c_2(\cV)=2$, or $\sigma^*\cU$, where $\cU$ is a null--correlation bundle on $\p3$.

In the first case let $\widehat{\pi}$ be the restriction of $\pi$ to $E$. It follows that $\cE\otimes\cO_E\cong\widehat{\pi}^*\cV$. Since $\pi$ is induced by $\cO_F(f)$, it follows that $\widehat{\pi}$ is the isomorphism induced by $\cO_F(f)\otimes\cO_E$. Thus the $\mu$--stability of $\cV$ with respect to $\cO_{\p2}(1)$ yields the $\mu$--stability of $\cE\otimes\cO_E$ with respect to $\cO_F(f)\otimes\cO_E\cong\widehat{\pi}^*\cO_{\p2}(1)$.

In particular $\cE$ is generically trivial on $\Lambda_E$ thanks to \cite[Corollary 2 of Theorem II.2.1.4]{O--S--S}. Moreover, $\cV$ is $\mu$--stable, hence
$$
h^1\big(F,\cE(-2\xi)\big)=h^2\big(F,\cE(-2f)\big)=h^2\big(\p2,\cV(-2)\big)=h^0\big(\p2,\cV(-1)\big)=0.
$$
thanks to Equality \eqref{Serre} and  \cite[Exercises III.8.1, III.8.3 and III.8.4]{Ha2}. It follows that $\cE$ is earnest thanks to Corollary \ref{cSimplify}. Finally, if $L\in\Lambda_\pi$ is general, then  $L=\pi^{-1}(x)$ for a general point $x\in\p2$. Thus, $\cO_L\cong\pi^*\bC_x$, where $\bC_x$ is the  skyscraper sheaf supported on $x$, hence
$$
\cE\otimes\cO_L\cong\pi^*(\cV\otimes \bC_x)\cong\pi^*(\bC_x^{\oplus2})\cong\cO_{\p1}^{\oplus2}.
$$

Now let $\cE\cong\sigma^*\cU$ be the pull--back of a general null--correlation bundle $\cU$ on $\p3$: thus $c_2(\cE)=\xi^2$, whence $c_2(\cE)h=2$, i.e. $\cE$ is an instanton bundle with minimal charge. Since $E=\sigma^{-1}(P)$, it follows that $\cO_E\cong\sigma^*\bC_P$, where $\bC_P$ is the  skyscraper sheaf supported on $P$. Thus
$$
\cE\otimes\cO_E\cong\sigma^*(\cU\otimes \bC_P)\cong\cO_{\p2}^{\oplus2},
$$
hence $\cE\otimes\cO_L\cong\cO_{\p1}^{\oplus2}$ for each $L\in\Lambda_E$, i.e. $\cE$ is generically trivial on $\Lambda_E$. Moreover, $h^1\big(F,\cE(-2\xi)\big)=0$, hence $\cE$ is earnest thanks to Corollary \ref{cSimplify}. 

If $\ell$ is a general line through $P$, then $\cU\otimes\cO_\ell\cong\cO_{\p1}^{\oplus2}$. The strict transforms of the lines through $P$ are the lines $L\in\Lambda_\pi$. The properties of the blow up map then imply that $\sigma_{\vert L}\colon L\to\ell$ is the blow up of $\ell$ at the point $P$, hence it is an isomorphism: in particular $\sigma^*\cO_\ell\cong\cO_L$, whence $\cE\otimes\cO_L\cong\sigma^*(\cU\otimes\cO_{\ell})\cong\cO_{\p1}^{\oplus2}$: i.e. $\cE$ is generically trivial on $\Lambda_\pi$ too. 
\end{proof}

On the one hand, by definition, instanton bundles on Fano threefolds $F$ with $\varrho_F=1$ are $\mu$--stable as pointed out in Example \ref{eFanoCyclic}. On the other hand, we showed in the same example that there exist strictly $\mu$--semistable instanton bundles on threefolds $F\subseteq\p7$ of degree $6$ and $i_F=2$. It is then natural to deal with the $\mu$--stability of instanton bundles using the following helpful lemma.

\begin{lemma}
\label{lHoppe}
Let $\cG$ be a rank $2$ vector bundle on $F$.

Then $\cG$ is $\mu$--stable (resp. $\mu$--semistable) if and only if  $h^0\big(F,\cG(-a\xi-b f)\big)=0$ for each $a,b \in\bZ$ such that $4a+3b \ge\mu(\cG)$ (resp. $>\mu(\cG)$).
\end{lemma}
\begin{proof}
The group $\Pic(F)$ is generated by the classes of $\xi$ and $f$, hence it suffices to apply \cite[Corollary 4]{J--M--P--S}.
\end{proof}

The following result could provide some evidence that instanton bundles on $F$ should be $\mu$--stable.

\begin{proposition}
\label{pStable}
Let $\cE$ be an instanton bundle with charge $\alpha\xi^2+\beta f^2$ on $F$.

If $\alpha\le14$, then $\cE$ is $\mu$--stable.
\end{proposition}
\begin{proof}
Let $\cE$ be a strictly $\mu$--semistable with $c_2(\cE)=\alpha\xi^2+\beta f^2$. 
Lemma \ref{lHoppe} yields that $h^0\big(F,\cE(-a\xi-bf)\big)\ne0$ for some integers $a,b$ such that $4a+3b=\mu(\cE)=0$: In particular there is $\lambda\in\bZ\setminus \{\ 0\ \}$ such that $a=3\lambda$, $b=-4\lambda$.

Thus there is a non--zero section $s\in H^0\big(F,\cE(-3\lambda\xi+4\lambda f)\big)$: the zero locus of $s$ is the union of a possibly empty subscheme $C\subseteq F$ of pure dimension $1$ and of a divisor $S\subseteq F$ which is either zero or effective. Thus $H^0\big(F,\cE(-3\lambda\xi+4\lambda f-S)\big)$ must contain a non--zero section vanishing exactly along $C$, hence we would have an injective map  $\cO_S(a\xi+bf+S)\to\cE$: it follows that $0\le Sh^2=(3\lambda\xi-4\lambda f+S)h^2\le\mu(\cE)=0$, i.e. $S=0$. 

In particular we have an exact sequence of the form
\begin{equation*}
 0 \longrightarrow\cO_F(3\lambda\xi-4\lambda f) \longrightarrow \cE \longrightarrow \cI_{C\vert F}(-3\lambda\xi+4\lambda f) \longrightarrow 0.
\end{equation*}
A simple computation shows that $C$ has class $(\alpha-15\lambda^2)\xi^2+(\beta+16\lambda^2)f^2$ in $A^2(F)$. Thus, for a general line $\ell\subseteq\p2$, the surface $F_0=\pi^{-1}(\ell)\in \vert f\vert$  intersects $C$ in a $0$--dimensional scheme of degree $\alpha-15\lambda^2\ge0$, hence $\alpha\ge15$.
\end{proof}

We close this section by proving that in our case it is not possible to construct instanton bundles by taking suitable extensions of line bundles on $F$ as in \cite{M--M--PL}. 

\begin{proposition}
\label{pExtension}
No instanton $\cE$ on $F$ can fit in an extension of the form
\begin{equation}
\label{seqExtension}
 0 \longrightarrow\cO_F(-a\xi-bf) \longrightarrow \cE \longrightarrow \cO_F(a\xi+bf) \longrightarrow 0.
\end{equation}
\end{proposition}
\begin{proof}
Assume that $\cE$ fitting in Sequence \eqref{seqExtension} actually exists. 

If the sequence splits, then
$$
\cE\cong\cO_F(-a\xi-bf) \oplus \cO_F(a\xi+bf).
$$
Since $\cE$ is $\mu$--semistable, it follows that both $\mu( \cO_F(a\xi+bf))$ and $\mu( \cO_F(-a\xi-bf))$ are not positive. Thus,
$$
4a+3b=\mu( \cO_F(a\xi+bf))=-\mu( \cO_F(-a\xi-bf))=0,
$$
hence there is $\lambda\in\bZ$ such that $a=3\lambda$, $b=-4\lambda$. We have $\lambda\ne0$, otherwise $h^0\big(F,\cE\big)=2$, contradicting the definition of instanton. If $\lambda\ne0$, we can assume that $\lambda\ge1$, hence Proposition \ref{pLineBundle} implies
$$
h^1\big(F,\cE(-h)\big)\ge h^1\big(F,\cO_F((-3\lambda-1)\xi+(4\lambda-1)f)\big)\ne0,
$$
again a contradiction.

From now on we will assume that Sequence \eqref{seqExtension} does not split: it follows that $h^1\big(F,\cO_F(-2a\xi-2bf)\big) \ne 0$. By Proposition \ref{pLineBundle}, this implies that $2a \ge 2$ and $2b\le-1$, which implies that $a \ge 1$ and $b \le -1$.

From Sequence \eqref{seqExtension}, it can be quickly seen that $c_2(\cE) = -b^2 f^2 - (2ab + a^2) \xi^2$. Thanks to Corollary \ref{cChargeBlow} we deduce that $-(b^2 +2ab + a^2)$ is non--negative, hence $b=-a$. The same proposition also implies $-4ab-2a^2-b^2\ge2$, hence $a\ge 2$.

The cohomology of Sequence \eqref{seqExtension} tensored by $\cO_F(-h)$, the above restriction on $a$ and $b$  then yields $h^0\big(F,\cO_F((a-1)\xi+(-a-1)f)\big)=0$, thanks to Proposition \ref{pLineBundle}, hence
$$
h^1\big(F,\cE(-h)\big)\ge h^1\big(F,\cO_F((-a-1)\xi+(a-1)f)\big).
$$
Again Proposition \ref{pLineBundle} implies that the dimension on the left is positive because $a-1\le-3$ and $-a-1\ge1$, hence $\cE$ cannot be an instanton.
\end{proof}

\section{The proof of Theorem \ref{tInstanton}}
\label{sProof}
In this section we will prove Theorem \ref{tInstanton}, showing the existence of earnest instanton bundles on $F$. To this purpose we will construct irreducible families of instanton bundles from very reducible curves using Theorem \ref{tSerre}. First we introduce an important family of conics on $F$.

\begin{remark}
\label{rConic}
Let us consider the curves $C$ in $F$ which are inverse images via $\sigma$ of lines not containing the blown up point $P$. Their class in $A^2(F)$ is $\xi^2$, hence $C$ is a conic and the ideal sheaf $\cI_{C\vert F}$ has the Koszul--type resolution
\begin{equation*}
0\longrightarrow\cO_F(-2\xi)\longrightarrow\cO_F(-\xi)^{\oplus2}\longrightarrow\cI_{C\vert F}\longrightarrow0.
\end{equation*}
In particular $\cN_{C\vert F}\cong\cO_{\p1}(1)^{\oplus2}$.

Let $\Gamma$ be the locus in the Hilbert scheme of conics corresponding to such curves. By construction $\Gamma$ is isomorphic to an open subset of the grassmannian of lines in $\p3$, hence it is irreducible of dimension $4$.
\end{remark}

We are now ready to prove Theorem \ref{tInstanton}. We start by explaining how to construct some particular instanton on $F$.

\begin{construction}
\label{conInstanton}
Let $\alpha\ge0$ and $\beta\ge0$ integers such that and $2\alpha+\beta\ge2$. 

Let $L_1,\dots,L_{\beta+1}\in\Lambda_\pi$ and $C_1,\dots, C_{\alpha}\in \Gamma$ be general curves. Notice that $C_i\cong L_j\cong \p1$ and we can assume such curves pairwise disjoint, due to the definition of $\Gamma$ and $\Lambda_\pi$, because both $\cO_F(\xi)$ and $\cO_F(f)$ are globally generated. We define 
\begin{equation}
\label{XInstanton}
X:=\bigcup_{i=1}^{\alpha}C_i\cup\bigcup_{j=1}^{\beta+1}L_j\subseteq F.
\end{equation}

We claim that $\det(\cN_{X\vert F})\cong\cO_F(2f)\otimes\cO_X$. Such an isomorphism can be checked component by component. Thanks to Remarks \ref{rLine} and \ref{rConic} we have 
\begin{gather*}
\det(\cN_{X\vert F})\otimes\cO_{C_i}\cong\cO_{\p1}(2)\cong\cO_F(2f)\otimes\cO_{C_i},\\
\det(\cN_{X\vert F})\otimes\cO_{L_j}\cong\cO_{\p1}\cong\cO_F(2f)\otimes\cO_{L_j}.
\end{gather*}
Theorem \ref{tSerre} yields the existence of a vector bundle $\cF$ on $F$ with a section $s$ vanishing exactly along $X$ and with $c_1(\cF)=2f$, $c_2(\cF)=X$, because the curves $C_i$ and $L_j$ have been chosen pairwise disjoint and $h^2\big(F,\cO_F(-2f)\big)=0$. Thus $\cE:=\cF(-f)$ is a vector bundle fitting into an exact sequence of the form
\begin{equation}
\label{seqE}
0\longrightarrow\cO_F(-f)\longrightarrow\cE\longrightarrow\cI_{X\vert F}(f)\longrightarrow0.
\end{equation}
Since $h^1\big(F,\cO_F(-2f)\big)=0$, then the bundle $\cE$ is uniquely determined by the scheme $X$.
\end{construction}

We are finally ready to prove Theorem \ref{tInstanton} stated in the introduction.

\medbreak
\noindent{\it Proof of Theorem \ref{tInstanton}.}
By construction $c_1(\cE)=0$, $c_2(\cE)=\alpha\xi^2+\beta f^2$. 

Since the image via $\sigma$ of an element in $\vert f\vert$, i.e. of a fibre of $\pi$ is a plane through the blown up point $P$, if $\alpha\ge1$ then $h^0\big(F,\cI_{X\vert F}(f)\big)=0$. If $\alpha=0$, then the restriction $2\alpha+\beta\ge2$ forces the existence of $\beta+1\ge3$ general curves in $\Lambda_\pi$. Since such curves are general, we deduce that their images via $\pi$ do not lie on the same line, hence again $h^0\big(F,\cI_{X\vert F}(f)\big)=0$. In both the cases $h^0\big(F,\cE\big)=0$ and the cohomology of Sequence \eqref{seqE} implies $h^1\big(F,\cE(-h)\big)= h^1\big(F,\cI_{X\vert F}(-\xi)\big)$. 

For each connected component $Y\cong\p1$ of $X$ we have $-\xi Y=-1$: it follows that $h^0\big(X,\cO_F(-\xi)\otimes\cO_X\big)=0$. The cohomology of the exact sequence
\begin{equation}
\label{seqIdeal}
0\longrightarrow\cI_{X\vert F}\longrightarrow\cO_F\longrightarrow\cO_X\longrightarrow0
\end{equation}
tensored by $\cO_F(-\xi)$ then yields $h^1\big(F,\cI_{X\vert F}(-\xi)\big)\le h^1\big(F,\cO_F(-\xi)\big)=0$, hence $h^1\big(F,\cE(-h)\big)=0$.

We will now show that $\cE$ is $\mu$--stable. To this purpose we will make use of Lemma \ref{lHoppe}, proving that  if $4a+3b=\mu(\cO_F(a\xi+bf))\ge\mu(\cE)=0$, i.e.
\begin{equation}
\label{AB}
b\ge-\frac43a
\end{equation}
then the cohomology of Sequence \eqref{seqE} tensored by $\cO_F(-a\xi-b f)$, i.e.
$$
0\longrightarrow\cO_F(-a\xi-(b+1)f)\longrightarrow\cE(-a\xi-bf)\longrightarrow\cI_{X\vert F}(-a\xi-(b-1)f)\longrightarrow0,
$$
returns $h^0\big(F,\cE(-a\xi-bf)\big)=0$. If $a\ge1$ such a vanishing is trivial. If $a=0$, then Inequality \eqref{AB} implies $b\ge0$, and again $h^0\big(F,\cE(-a\xi-b f)\big)=0$ trivially for $b\ge2$. The same vanishing holds also for $0\le b\le1$, because $X$ contains at least two disjoint components. 

We restrict our attention to the case $a\le-1$. In this case Inequality \eqref{AB} implies
$$
-a-(b+1)\le-1+\frac13a\le-1,
$$
hence $h^0\big(F,\cO_F(-a\xi-b f)\big)=0$. Inequality  \eqref{AB} yields $-a-b\le a/3$. If $a\le -4$ this implies $-a-b\le-2$ hence $-a-(b-1)\le-1$. The same is true if $-3\le a\le -1$ and $b\ge2-a$. Thus
$$
h^0\big(F,\cI_{X\vert F}(-a\xi-(b-1)f)\big)\le h^0\big(F,\cO_F(-a\xi-(b-1)f)\big)=0.
$$
It remains to deal with the case $-3\le a\le -1$ and $b=1-a$: in this case
$$
h^0\big(F,\cI_{X\vert F}(-a\xi-(b-1)f)\big)=h^0\big(F,\cI_{X\vert F}(-a(\xi-f))\big).
$$
The support of each divisor in $\vert \lambda(\xi-f)\vert$ with $\lambda\ge1$ is $E$, hence for a general choice of $X$ the dimension on the right is zero.

We prove that $\cE$ is earnest: thanks to Corollary \ref{cSimplify} it suffices to check that $h^1\big(F,\cE(-2\xi)\big)=0$. The cohomology of Sequence \eqref{seqE} tensored by $\cO_F(-2\xi)$ yields the equality
$$
h^1\big(F,\cE(-2\xi)\big)=h^1\big(F,\cI_{X\vert F}(-2\xi+f)\big).
$$
We now show that $h^1\big(F,\cI_{X\vert F}(-2\xi+f)\big)=0$. 

To this purpose it suffices to consider the cohomology of Sequence \eqref{seqIdeal} tensored by $\cO_F(-2\xi+f)$ and to prove that the induced map
$$
\psi\colon H^1\big(F,\cO_F(-2\xi+f)\big)\to H^1\big(X,\cO_X\otimes\cO_F(-2\xi+f)\big)
$$
is injective. Since $h^1\big(F,\cO_F(-2\xi+f)\big)=1$ thanks to Proposition \ref{pLineBundle}, it follows that it suffices to check that $\psi$ is non--zero, i.e. $\ker(\psi)=0$. 

Let $X_0:=\bigcup_{j=1}^{\beta+1}L_j$. Taking into account that
\begin{gather*}
\cO_F(-2\xi+f)\otimes\cO_{C_i}\cong\cO_{\p1}(-1),\\
\cO_F(-2\xi+f)\otimes\cO_{L_j}\cong\cO_{\p1}(-2),
\end{gather*}
we deduce that $\psi$ coincides with the map induced in cohomology by the restriction
$\cO_F(-2\xi+f)\to\cO_{X_0}\otimes\cO_F(-2\xi+f)$, hence $\ker(\psi)\subseteq H^1\big(F,\cI_{X_0\vert F}(-2\xi+f)\big)$.

We have $X_0=\pi^{-1}(Z)$ where $Z\subseteq\p2$ is a general $0$--dimensional subscheme of degree $\beta+1$. Being $0$--dimensional, $Z$ is arithmetically Cohen--Macaulay, hence there is an exact sequence
$$
0\longrightarrow\bigoplus_j\cO_{\p2}(-t_j)\longrightarrow\bigoplus_i\cO_{\p2}(-s_j)\longrightarrow\cI_{Z\vert \p2}\longrightarrow0
$$
where $s_i,t_j$ are positive integers. The flatness of $\pi$ yields the existence of an exact sequence 
$$
0\longrightarrow\bigoplus_j\cO_{F}(-2\xi+(1-t_j)f)\longrightarrow\bigoplus_i\cO_{F}(-2\xi+(1-s_i)f)\longrightarrow\cI_{X_0\vert F}(-2\xi+f)\longrightarrow0.
$$
We deduce that $h^1\big(F,\cI_{X_0\vert F}(-2\xi+f)\big)$ thanks to Proposition \ref{pLineBundle}, because $1-s_i\le0$, hence $\ker(\psi)=0$. Thus the claimed vanishing $h^1\big(F,\cI_{X\vert F}(-2\xi+f)\big)=0$ holds true.

We prove that $\cE$ is generically trivial. Since $\cE$ is earnest, it follows from Lemma \ref{lEarnest} that it is generically trivial on $\Lambda_E$. Thus we have to show that it is generically trivial on $\Lambda_\pi$. If $M\in\Lambda_\pi$ is general, then $Mf=0$ and $M\cap X=\emptyset$, hence Sequence \eqref{seqE} tensored by $\cO_M\cong\cO_{\p1}$ becomes
$$
0\longrightarrow\cO_{\p1}\longrightarrow\cE\otimes\cO_M\longrightarrow\cO_{\p1}\longrightarrow0,
$$
hence $\cE\otimes\cO_M\cong\cO_{\p1}^{\oplus2}$.

We finally prove the assertion on the dimensions of the $\Ext$ groups. Since $\cE$ is $\mu$--stable, then it is simple, hence $\Ext^3_{F}\big(\cE,\cE\big)=0$ thanks to Lemma \ref{lExt3}. We will show below that $\Ext^2_{F}\big(\cE,\cE\big)\cong H^2\big(F,\cE\otimes\cE^\vee\big)=0$, hence 
$$
\dim\Ext^1_{F}\big(\cE,\cE\big)=8\alpha+4\beta-3,
$$
thanks to Lemma \ref{lExt12}. To this purpose, the cohomology of Sequence \eqref{seqE} tensored by $\cE^\vee\cong\cE$ returns
$$
H^2\big(F,\cE(-f)\big)\longrightarrow H^2\big(F,\cE\otimes\cE^\vee\big)\longrightarrow H^2\big(F,\cE\otimes\cI_{X\vert F}(f)\big),
$$
hence it suffices to check that $h^2\big(F,\cE(-f)\big)=h^2\big(F,\cE\otimes\cI_{X\vert F}(f)\big)=0$.

We first check that $h^2\big(F,\cE(-f)\big)=0$. Indeed, thanks to Proposition \ref{pLineBundle} the cohomologies of Sequences \eqref{seqIdeal} and \eqref{seqE} tensored by $\cO_F(-f)$ return
$$
h^2\big(F,\cE(-f)\big)\le h^2\big(F,\cI_{X\vert F}\big)\le h^1\big(F,\cO_X\big).
$$
The dimension on the right is zero, because $X$ is the disjoint union of smooth rational curves. 

Finally we check that $h^2\big(F,\cE\otimes\cI_{X\vert F}(f)\big)=0$. The cohomology of Sequence \eqref{seqE} tensored by $\cO_F(f)$ implies $h^2\big(F,\cE(f)\big)\le h^2\big(F,\cI_{X\vert F}(2f)\big)$. The cohomology of Sequence \eqref{seqIdeal} tensored by $\cO_F(2f)$ then yields
$$
h^2\big(F,\cI_{X\vert F}(2f)\big)\le h^1\big(X,\cO_F(2f)\otimes\cO_X\big).
$$
Since $\cO_F(2f)$ restricts to each component of $X$ to a line bundle of degree $2$, it follows that the dimension on the right is zero. In particular $h^2\big(F,\cI_{X\vert F}(2f)\big)=0$, hence the cohomology of Sequence \eqref{seqE} tensored by $\cO_F(f)$ implies $h^2\big(F,\cE(f)\big)=0$. Hence the cohomology of Sequence \eqref{seqIdeal} tensored by $\cE(f)$ returns 
\begin{align*}
h^2\big(F,\cE\otimes\cI_{X\vert F}(f)\big)&\le h^1\big(X,\cE(f)\otimes\cO_X\big)=\\
&=\sum_{i=1}^{\alpha}h^1\big(X,\cE(f)\otimes\cO_{C_i}\big)+\sum_{j=1}^{\beta+1}h^1\big(X,\cE(f)\otimes\cO_{L_j}\big).
\end{align*}
Equality \eqref{Normal} and the definition of $\cE$ imply $\cE(f)\otimes\cO_X\cong\cN_{X\vert F}$. Thus 
\begin{gather*}
\cE(f)\otimes\cO_{C_i}\cong\cO_{\p1}(1)^{\oplus2},\qquad 
\cE(f)\otimes\cO_{L_j}\cong\cO_{\p1}^{\oplus2},
\end{gather*}
thanks to Remarks \ref{rLine} and \ref{rConic}, hence $h^2\big(F,\cE\otimes\cI_{X\vert F}(f)\big)=0$.
\qed
\medbreak

\section{Some comments on the irreducibility of $\cI_F(\alpha\xi^2+\beta f^2)$}
\label{sComponent}
Recall that $\cI_F(\alpha\xi^2+\beta f^2)$ denotes the locus of points representing instanton bundles inside $\cM_F(2;0,\alpha\xi^2+\beta  f^2)$. 

Theorem \ref{tInstanton} shows that the locus of points $\cI_F^{earnest}(\alpha\xi^2+\beta f^2)$ representing earnest instantons contains at least a non--empty irreducible component. In this last section we will discuss about the irreducibility of such loci.

We start by giving the following complete characterisations of earnest instanton bundles with charges $\alpha\xi^2$ and $\beta f^2$ in terms of bundles on $\p3$ and $\p2$.

\begin{proposition}
\label{pB=0}
Let $\cE$ be an earnest instanton bundle with charge $\alpha\xi^2$ on $F$. 

Then $\cE\cong\sigma^*\cU$, where $\cU$ is an instanton bundle with charge $\alpha$ on $\p3$.
\end{proposition}
\begin{proof}
Recall that an instanton bundle $\cE$ is earnest if and only if $h^1\big(F,\cE(-2\xi)\big)=0$ thanks to Corollary \ref{cSimplify}.

Both $\cE(-2\xi)$ and $\cE(-2f)$ are subbundles of $\cE$, hence
$$
h^0\big(F,\cE(-2\xi)\big)=0,\qquad h^3\big(F,\cE(-2\xi)\big)=h^0\big(F,\cE(-2f)\big)=0.
$$ 
Equalities \eqref{Exotic} and \eqref{RRBlowUp} imply $\chi(\cE(-2\xi))=0$, hence $h^2\big(F,\cE(-2\xi)\big)=0$. Thus the cohomology of the exact sequence
$$
0\longrightarrow \cO_F(f-\xi)\longrightarrow \cO_F\longrightarrow \cO_E\longrightarrow0
$$
tensored by $\cE(-h)$ and Equalities \eqref{VanishingQ}, \eqref{Serre} yield
$$
h^1\big(E,\cE(-h)\otimes\cO_E\big)=h^2\big(E,\cE(-2h)\otimes\cO_E\big)=0,
$$
i.e. $\cE\otimes\cO_E$ is $0$--regular, hence globally generated. Since $c_1(\cE\otimes\cO_E)=0$, it follows that $\cE\otimes\cO_E\cong\cO_E^{\oplus2}$, hence $\cU:=\sigma_*\cE$ is a vector bundle of rank $2$ on $\p3$ such that $\sigma^*\cU\cong\cE$ (see \cite[Theorem 5]{Sch}).

In particular $c_1(\cU)=0$ and $c_2(\cU)=\alpha$ necessarily. Since $\sigma_*\cO_F\cong\cO_{\p3}$ and $R^i\sigma_*\cO_F=0$ for $i\ge1$, it follows from \cite[Exercises III.8.1 and III.8.3]{Ha2} that 
$h^0\big(\p3,\cU\big)=h^0\big(F,\cE\big)=0$. The same argument also implies $h^1\big(\p3,\cU(-2)\big)=h^1\big(F,\cE(-2\xi)\big)=0$, because  $\sigma^*(\cU(-2))=\cE(-2\xi)$. We deduce that $\cU$ is an instanton bundle on $\p3$.
\end{proof}

The following proposition gives a complete characterisation of instanton bundles with charge $\beta f^2$, because all such bundles are automatically earnest thanks to Theorem \ref{tSimplify} and Corollary \ref{cSimplify}.

\begin{proposition}
\label{pA=0}
Let $\cE$ be an instanton bundle  with charge $\beta f^2$ on $F$.

Then $\cE\cong\pi^*\cV$, where $\cV$ is a $\mu$--stable bundle on $\p2$ with $c_1(\cV)=0$ and $c_2(\cV)=\beta$.
\end{proposition}
\begin{proof}
Thanks to Theorem \ref{tSimplify} the instanton bundle $\cE$ is the cohomology of a monad
$$
0\longrightarrow\cO_F(-f)^{\oplus\beta}\longrightarrow\pi^*\Omega_{\p2}^1(f)^{\oplus\beta}\longrightarrow\cO_F^{\oplus\beta-2}\longrightarrow0.
$$
Since $\cO_F(f)\cong\pi^*\cO_{\p2}(1)$ and the map $\pi$ is flat, such a monad is the pull--back of a monad on $\p2$ of the form
$$
0\longrightarrow\cO_{\p2}(-1)^{\oplus\beta}\longrightarrow\Omega_{\p2}^1(1)^{\oplus\beta}\longrightarrow\cO_{\p2}^{\oplus\beta-2}\longrightarrow0.
$$
Thus, if $\cV$ is the cohomology of the latter, then $\cE\cong\pi^*\cV$. In particular $\cV$ is necessarily locally free.

Using the above monad, it is not difficult to check that $\cV$ satisfies the conditions in the statement.
\end{proof}

Recall that $\cM_{\p2}(2;0,\beta)$ is integral, rational and smooth (see \cite{Ba}: see also \cite{Mar2,Ma,Ka}). It then follows from Proposition \ref{pA=0}, that the locus
$$
\cI_F^{earnest}(\beta f^2)=\cI_F(\beta f^2)\subseteq\cM_F(2;0,\beta f^2)
$$
of earnest instanton bundles is integral, rational and smooth as well. 

When $\beta=0$ we do not know whether $\cI_F(\alpha\xi^2)$ is still irreducible: anyhow we deduce from \cite{J--V,Tik1,Tik2} that  the locus $\cI_F^{earnest}(\alpha\xi^2)$ is irreducible and smooth thanks to Proposition \ref{pB=0}.

In general, we are not even able to prove the irreducibility of $\cI_F^{earnest}(\alpha\xi^2+\beta f^2)$. Nevertheless, we now  prove Theorem \ref{tComponent}, i.e. that Construction \ref{conInstanton} always leads to bundles in the same component.

\medbreak
\noindent{\it Proof of Theorem \ref{tComponent}.}
The schemes as in Equality \eqref{XInstanton} represent points in a non--empty open subset $\mathcal H\subseteq\Gamma^{\times\alpha}\times \Lambda_\pi^{\times\beta+1}$. Since the latter product is irreducible, it follows that the same holds for $\mathcal H$. 

Since the  bundle $\cE$ in Sequence \eqref{seqE} is uniquely determined by the scheme $X$, we obtain in this way a flat family of bundles  containing all the bundles obtained via Construction \ref{conInstanton} and parameterized by $\mathcal H$. Thus we deduce the existence of a morphism $u\colon \mathcal H\to \cI_F(\alpha\xi^2+\beta  f^2)$. Every point in $u(\mathcal H)$ is smooth because $\Ext^2_{F}\big(\cE,\cE\big)=0$ (see Theorem \ref{tInstanton}), thus there is a unique component $\cI_F^{0}(\alpha\xi^2+\beta  f^2)$ containing $u(\mathcal H)$: Theorem \ref{tInstanton} then implies
$$
\dim\cI_F^{0}(\alpha\xi^2+\beta  f^2)=\dim \Ext^1_{F}\big(\cE,\cE\big)=8\alpha+4\beta-3.
$$
This last equality completes the proof of the statement.
\qed
\medbreak

Thanks to Theorem \ref{tInstanton}, we know that each bundle in $\cI_F^{0}(\alpha\xi^2+\beta  f^2)$ is earnest, then $\cI_F^{0}(\alpha\xi^2+\beta  f^2)\subseteq \cI_F^{earnest}(\alpha\xi^2+\beta  f^2)$. The above discussion shows that equality holds when $\alpha$ or $\beta$ vanishes.

\bigskip
\noindent
Gianfranco Casnati,\\
Dipartimento di Scienze Matematiche, Politecnico di Torino,\\
c.so Duca degli Abruzzi 24,\\
10129 Torino, Italy\\
e-mail: {\tt gianfranco.casnati@polito.it}

\bigskip
\noindent
Emre Coskun,\\
Department of Mathematics, Middle East Technical University,\\
06800, Ankara, Turkey\\
e-mail: {\tt emcoskun@metu.edu.tr}

\bigskip
\noindent
Ozhan Genc,\\
Dipartimento di Scienze Matematiche, Politecnico di Torino,\\
c.so Duca degli Abruzzi 24,\\
10129 Torino, Italy\\
e-mail: {\tt ozhangenc@gmail.com}

\bigskip
\noindent
Francesco Malaspina,\\
Dipartimento di Scienze Matematiche, Politecnico di Torino,\\
c.so Duca degli Abruzzi 24,\\
10129 Torino, Italy\\
e-mail: {\tt francesco.malaspina@polito.it}

\end{document}